\documentclass[11pt]{amsart}
\usepackage{amscd}
\usepackage[arrow,matrix]{xy}
\usepackage{graphicx}
\usepackage{amsmath}
\usepackage{amsmath, latexsym, amssymb}
\input xypic
\numberwithin{equation}{section}
\theoremstyle{plain}
\newtheorem{lemma}{Lemma}[section]
\newtheorem{proposition}[lemma]{Proposition}
\newtheorem{theorem}[lemma]{Theorem}
\newtheorem{corollary}[lemma]{Corollary}

\theoremstyle{definition}
\newtheorem{definition}[lemma]{Definition}
\newtheorem{remark}[lemma]{Remark}
\newtheorem{example}[lemma]{Example}

\begin{document}
\newcommand{\R}{{\mathbb R}}
\newcommand{\C}{{\mathbb C}}
\newcommand{\F}{{\mathbb F}}
\renewcommand{\O}{{\mathbb O}}
\newcommand{\Z}{{\mathbb Z}} 
\newcommand{\N}{{\mathbb N}}
\newcommand{\Q}{{\mathbb Q}}
\renewcommand{\H}{{\mathbb H}}

\newcommand{\Aa}{{\mathcal A}}
\newcommand{\Bb}{{\mathcal B}}
\newcommand{\Cc}{{\mathcal C}}    
\newcommand{\Dd}{{\mathcal D}}
\newcommand{\Ee}{{\mathcal E}}
\newcommand{\Ff}{{\mathcal F}}
\newcommand{\Gg}{{\mathcal G}}    
\newcommand{\Hh}{{\mathcal H}}
\newcommand{\Kk}{{\mathcal K}}
\newcommand{\Ii}{{\mathcal I}}
\newcommand{\Jj}{{\mathcal J}}
\newcommand{\Ll}{{\mathcal L}}    
\newcommand{\Mm}{{\mathcal M}}    
\newcommand{\Nn}{{\mathcal N}}
\newcommand{\Oo}{{\mathcal O}}
\newcommand{\Pp}{{\mathcal P}}
\newcommand{\Qq}{{\mathcal Q}}
\newcommand{\Rr}{{\mathcal R}}
\newcommand{\Ss}{{\mathcal S}}
\newcommand{\Tt}{{\mathcal T}}
\newcommand{\Uu}{{\mathcal U}}
\newcommand{\Vv}{{\mathcal V}}
\newcommand{\Ww}{{\mathcal W}}
\newcommand{\Xx}{{\mathcal X}}
\newcommand{\Yy}{{\mathcal Y}}
\newcommand{\Zz}{{\mathcal Z}}

\newcommand{\zt}{{\tilde z}}
\newcommand{\xt}{{\tilde x}}
\newcommand{\Ht}{\widetilde{H}}
\newcommand{\ut}{{\tilde u}}
\newcommand{\Mt}{{\widetilde M}}
\newcommand{\Llt}{{\widetilde{\mathcal L}}}
\newcommand{\yt}{{\tilde y}}
\newcommand{\vt}{{\tilde v}}
\newcommand{\Ppt}{{\widetilde{\mathcal P}}}
\newcommand{\bp }{{\bar \partial}} 

\newcommand{\Remark}{{\it Remark}}
\newcommand{\Proof}{{\it Proof}}
\newcommand{\ad}{{\rm ad}}
\newcommand{\Om}{{\Omega}}
\newcommand{\om}{{\omega}}
\newcommand{\eps}{{\varepsilon}}
\newcommand{\Di}{{\rm Diff}}
\newcommand{\Pro}[1]{\noindent {\bf Proposition #1}}
\newcommand{\Thm}[1]{\noindent {\bf Theorem #1}}
\newcommand{\Lem}[1]{\noindent {\bf Lemma #1 }}
\newcommand{\An}[1]{\noindent {\bf Anmerkung #1}}
\newcommand{\Kor}[1]{\noindent {\bf Korollar #1}}
\newcommand{\Satz}[1]{\noindent {\bf Satz #1}}

\renewcommand{\a}{{\mathfrak a}}
\renewcommand{\b}{{\mathfrak b}}
\newcommand{\e}{{\mathfrak e}}
\renewcommand{\k}{{\mathfrak k}}
\newcommand{\pg}{{\mathfrak p}}
\newcommand{\g}{{\mathfrak g}}
\newcommand{\gl}{{\mathfrak gl}}
\newcommand{\h}{{\mathfrak h}}
\renewcommand{\l}{{\mathfrak l}}
\newcommand{\sm}{{\mathfrak m}}
\newcommand{\n}{{\mathfrak n}}
\newcommand{\s}{{\mathfrak s}}
\renewcommand{\o}{{\mathfrak o}}
\newcommand{\so}{{\mathfrak so}}
\renewcommand{\u}{{\mathfrak u}}
\newcommand{\su}{{\mathfrak su}}
\newcommand{\ssl}{{\mathfrak sl}}
\newcommand{\ssp}{{\mathfrak sp}}
\renewcommand{\t}{{\mathfrak t }}
\newcommand{\Cinf}{C^{\infty}}
\newcommand{\la}{\langle}
\newcommand{\ra}{\rangle}
\newcommand{\half}{\scriptstyle\frac{1}{2}}
\newcommand{\p}{{\partial}}
\newcommand{\notsub}{\not\subset}
\newcommand{\iI}{{I}}               
\newcommand{\bI}{{\partial I}}      
\newcommand{\LRA}{\Longrightarrow}
\newcommand{\LLA}{\Longleftarrow}
\newcommand{\lra}{\longrightarrow}
\newcommand{\LLR}{\Longleftrightarrow}
\newcommand{\lla}{\longleftarrow}
\newcommand{\INTO}{\hookrightarrow}

\newcommand{\QED}{\hfill$\Box$\medskip}
\newcommand{\UuU}{\Upsilon _{\delta}(H_0) \times \Uu _{\delta} (J_0)}
\newcommand{\bm}{\boldmath}

\title[Geometric structures associated with a  Cartan 3-form]{\large   Geometric structures  associated with
a simple Cartan 3-form}
\author{H\^ong V\^an L\^e} 

\medskip

\abstract We introduce the notion of a manifold admitting a simple compact Cartan 3-form $\om^3$.  We study algebraic types of such manifolds, in particular  those  having skew-symmetric torsion, or those  associated with a closed or coclosed  3-form $\om^3$. We prove the existence of an algebra of multi-symplectic forms $\phi^l$ on these manifolds. Cohomology groups associated with   complexes   of differential forms on $M^n$ in presence of such a closed multi-symplectic form $\phi^l$ and their relations with the de Rham cohomologies of $M$ are investigated. We show  rigidity of a class of strongly associative (resp. strongly coassociative) submanifolds. We include an appendix describing  all connected simply connected
 complete Riemannian manifolds admitting a parallel 3-form.
\endabstract
\maketitle
\tableofcontents

AMSC2010: 53C10, 53C2, 53C38.


\section{Introduction}

Let $G$ be a  Lie subgroup  of  $O(n)$. We want to characterize a class of ``natural" $G$-structures on  Riemannian manifolds $M^n$.  First, we  would like to see
$G$  in the list of  possible holonomy  groups of  Riemannian manifolds  $M^n$.  Second, we  also like to characterize
$G$ as  the    stabilizer group  of some exterior $k$-form on $\R^n$ (as it is the case with  most of  special holonomy
groups of Riemannian manifolds, see  \cite[table 1, chapter 10]{Besse1987}).  Note that  the holonomy group of a Riemannian  manifold  $M^n$  lies in such a 
``natural" group $G \subset  O(n)$  only if $M^n$ admits a parallel $k$-form $\phi^k$, which  makes $M^n$ a  manifold
with calibration.   
A careful analysis shows that  a  connected simply connected complete Riemannian manifold $M^n$ admits a parallel
3-form $\phi^3$, if and only if $(M^n, \phi^3)$ is a product of basis   Riemannian manifolds $(M^i, \phi_i ^3)$, where
either $\phi^3_i = 0$, or $M^i$ is flat and $\phi^3_i$ is a parallel form,  or $\phi^3_i$ is one of  stable 3-forms in dimensions 6,7,8, or $\phi^3$ is a wedge product of a K\"ahler 2-form with
a 1-form, or $\phi^3$ is a Cartan 3-form associated with a simple compact Lie group, see Theorem \ref{3para} for a precise  formulation.

This motivates  us to study geometry associated with a Cartan 3-form.  It turns out that these manifolds possess  very
rich geometric structures, arising from the cohomological structure of the associated Lie algebra.
 Our study can be thought as a continuation of   the study initiated by Hitchin of  geometries associated with  stable 3-forms  in dimension 6, 7, 8  \cite{Hitchin2000}, \cite{Hitchin2001}, which includes the  Special Lagrangian (SL) 3-form, the 3-form of $G_2$-type, and  the Cartan 3-form on $\su(3)$.
On the other hand, our study provides  new examples  and some structure theorems   for the theory of manifolds supplied with a closed multi-symplectic form,  which has been discovered long time ago  in relation with
the multi-variate field theory \cite{GS1977}, and enjoys its active development nowadays  \cite{BHR2010}, \cite{MS2010}.  

The plan of our  note is as follows.  In section 2 we  recall the definition of a  Cartan 3-form $\om _\g$ and show
that it is multi-symplectic if and only if $\g$ is semisimple  (Lemma \ref{multi}). We  compute  the stabilizer of $\om_\g$ in the case  that $\g$ is a simple Lie algebra over $\C$ or a real form
of a simple complex Lie algebra  (Theorem \ref{dyn1}). In section 3  we present many examples of  manifolds provided with a simple Cartan 3-form.
In section 4, using the notion of intrinsic torsion, we prove several structure theorems on algebraic types of a manifold provided with a 
compact simple Cartan form $\om^3$, especially  on those algebraic types having  skew-symmetric torsions, and those associated with  a closed or coclosed
3-form $\om ^3$ of type $\om_\g$ (Proposition \ref{main1}, Theorem \ref{main2}, Lemma \ref{kth}, Corollary \ref{skew2}). We end  this section with a  theorem describing  torsion-free complete $Aut(\g)$-manifolds  (Theorem \ref{tor}).  In section 5  we show the existence of an algebra of nowhere vanishing multi-symplectic forms $\phi^l$
on an orientable manifold  $M^n$ equipped with a compact Cartan 3-form  (Theorem \ref{main3} and Lemma \ref{exi2}). We study  cohomology groups associated
with a closed multi-symplectic form $\phi^l$  of the considered type and their relations with the de Rham cohomology groups of $M^n$ (Theorem \ref{spec2}, Example \ref{csu3}, Proposition \ref{har2}, Lemma \ref{har1}).  In section 6 we study   a class of strongly associative (or strongly coassociative) submanifolds in a manifold
$M^n$ provided with  a compact simple Cartan  3-form   and prove their algebraic and geometric rigidity  (Propositions \ref{main5},  \ref{triv4}, Proposition \ref{triv5}, Proposition \ref{triv6}, Remark \ref{liu})).  In section 7 we discuss some questions for further research.  In  the appendix we describe  simply-connected  complete Riemannian  manifolds admitting a parallel 
3-form.

\section{Cartan 3-form  $\om_\g$  and its stabilizer group}

In this section we recall the definition of  the Cartan 3-form  associated with a semisimple Lie algebra
$\g$. We show that  a Cartan 3-form is multi-sympletic  (Lemma \ref{multi}). We   compute
the stabilizer group of the Cartan  3-form in the case that $\g$ is a simple Lie algebra over $\C$, or
a real form of a  simple Lie algebra over $\C$, see Theorem \ref{dyn1}.  We discuss a generalization of Theorem \ref{dyn1}
in Remark \ref{hit}.3.
 
Let $\g$ be a semisimple Lie algebra  over $\C$ or over $\R$.  The Cartan 3-form $\om_\g$
is defined on $\g$ as follows
$$\om_\g (X, Y, Z) := \la X, [Y,Z]\ra ,$$
where $\la , \ra$ denotes the  Killing  bilinear form on $\g$.

Let  $\F$ be  field  $\C$ or $\R$. We recall that a $k$-form  $\om$ on $\F^n$ is {\it multi-symplectic},
if  the linear map
\begin{equation}
L_\om: \F ^n \to  \Lambda ^{k-1} ( \F^n ) ^*,  v \mapsto v \rfloor  \om ,\label{multi}
\end{equation}
is an injective map.

\begin{lemma}\label{multi}  The Cartan 3-form $\om _\g$ is multi-symplectic, if and only if $\g$ is semismple.
\end{lemma}

\begin{proof}  Assume that $\g$ is semisimple,  then $[\g, \g ] = \g$. Taking into account the non-degeneracy of  the  Killing bilinear form on $\g$, we get  immediately the ``if" assertion.   Now assume that $\om_\g$ is multi-simplectic.
Then the Killing form $\la, \ra$ is non-degenerate, since the kernel of the Killing form lies in the kernel 
of $L_{\om_\g}$. This  proves the ``only if" assertion.
\end{proof}

Next we note that  the stabilizer group $Stab (\om_\g)$ of $\om_\g$ contains the  automorphism  group $Aut (\g)$  of the Lie algebra $\g$.  


\begin{theorem}\label{dyn1}(cf. \cite{Freudenthal1953}, \cite[Theorem 7]{Kable2009}) 1. Let $\g$ be a simple  Lie algebra over $\C$. 
Then $Stab (\om_\g) =   Aut(\g)\times \Z_3$ if $\dim \g >3$.  If $\dim \g = 3$ then $Stab (\om_\g) =  SL(\g)$.\\
2. Let  $\g$ be a real form of  a complex simple Lie algebra over $\C$.  Then $Stab (\om _\g) =   Aut(\g)$, if
$\dim \g \ge 3$. If $\dim \g = 3$, then $Stab(\om_\g) = SL(\g)$.
\end{theorem}

\begin{proof}  Kable \cite[Theorem 7]{Kable2009}  showed that the stabilizer
group of the Cartan 3-form  $\om _\g$ on a semisimple Lie algebra  $\g$ is a  semi-direct product  of  $Aut(\g)$ with an abelian subgroup
$M(\g)$ consisting of elements $g \in  GL(\g)$ such that $g$ commutes with the adjoint action of $\g$, and $g^3 =1$.
Clearly Theorem \ref{dyn1} follows  from   Kable's theorem and  Schur's lemma.    

We  also  obtain Theorem \ref{dyn1} from  the  result by Freudenthal, which determined the identity component of the group  $Stab(\om_\g)$ \cite{Freudenthal1953}.   Denote by $\Nn _{GL(\g)} (\g)$ the normalizer of $\g$ in $GL (\g)$,
and by $\Zz_{GL(\g)} (\g)$  the centralizer of $\g$ in $GL (\g)$. We observe that  there is  a monomorphism
\begin{equation}
\Nn_{GL (\g)}  (\g)/\Zz _{GL (\g)}(\g)  \to Aut (\g) .
\label{aut}
\end{equation}
By  Freudenthal's  theorem  the group $Stab (\om_\g)$ is a subgroup of
$\Nn_{GL (\g)}(\g)$. Since the adjoint representation  of $\g$  is irreducible,     Schur's lemma  implies that
$\Zz_{GL (\g)}   (\g)$ is  the center  $\Zz ( GL(\g))$ of $GL (\g)$.
 Using (\ref{aut})  we get  the inclusion $Stab (\om_\g) \subset Aut (\g) \times \Z_3$,  where $\Z_3 = \Zz (GL (\g)) \cap Stab (\om _\g)$,  if $\g$ is a complex simple Lie algebra of dimension at least 8.
We  also get the inclusion $Stab (\om_\g) \subset Aut (\g)$, if $\g$ is a real form of a complex simple Lie algebra.
Now Theorem \ref{dyn1} follows immediately, observing that $Aut(\g) \subset Stab (\om _\g)$.
 \end{proof}

\begin{remark}\label{pan} In his  famous paper \cite{Dynkin1952}  Dynkin  explained  how one can apply  his result to     determine  the Lie algebra of the  stabilizer group of a $G$-invariant  differential form, where $G$ is a simple   Lie group.    Dynkin's  idea and  his classification result in \cite{Dynkin1952}  form   a main ingredient   of   Kable's proof (and our  unpublished  proof)  of Theorem \ref{dyn1}.
We thank   Dmitri Panyushev  for    informing us   of Kable's  paper and Freudenthal's  results  after a proof of Theorem \ref{dyn1} is obtained.
\end{remark}

\begin{remark}\label{hit}  1. Theorem \ref{dyn1}.2 is  a generalization of our theorem in \cite{LPV2008} for
case $\g = \ssl(3,\C)$  and its real forms, which has been proved by a different method.

2. Let $Ad(\g)$ denote the adjoint group of a Lie algebra $\g$. If $\g$ is a complex simple  Lie algebra or a compact form of a complex simple Lie algebra, then $Aut(\g)/Ad(\g)$ is  isomorphic to the  automorphism group  $Aut(D(\g))$ of  the Dynkin diagram  $D(\g)$ of $\g ^\C$, \cite[Theorem IX.5.4, Theorem IX.5.5, Theorem X.3.29]{Helgason1978}.  
It is well-known that  $Aut (\so (8)) = \Sigma_3$ -the permutation group on three letters, $Aut (D(\g)) =\Z_2$, if $\g = \su_n$ or $\g = \so (2n), \, n\not=4$,  or $\g = E_6$.  In other cases $Aut(D(\g)) = Id$.

3.  It  follows from  Dynkin results \cite[Theorem 2.1, Theorem 2.2]{Dynkin1952} that  if $\g$ is a  simple  complex Lie
algebra  over $\C$ or over $\R$ such that $\g \not = \ssp (n, \C)$  and $\g \not = \so (n, \C)$, then   the  Lie algebra of the stabilizer group of any $Ad(\g)$-invariant form  $\phi^l$ on $\g$ coincides with $\g$. We conjecture that this  assertion also holds for  $\g = \ssp (n)$ or $\g = \so (n)$.   If the Lie algebra of the stabilizer group  of $\phi^l$ is $\g$ we can use the same method in the proof  of Theorem \ref{dyn1}
to  find the stabilizer group of $\phi^l$.

4. Let $\g$ be a compact simple Lie group. The algebra of  $Ad(\g)$-invariant forms on $\g$ is  equal to  the algebra of de Rham cohomologies of any compact Lie group $G$ having Lie algebra $\g$.  
The algebra $H^*(G, \R)$ of a compact simple Lie group $G$ is generated by primitive elements $x_i, y_i$   of degree $i$ as follows \cite[App.A, Table 6.1.4, p.1742]{EDM1993}.\\
A.  $H^* (SU(n), \R) = \Lambda (x_3, x_5, \cdots, x_{2n-1})$.\\
B.  $H^* (SO(2n+1),\R) = \Lambda (x_3,x_7, \cdots, x_{4n-1})$.\\
C. $H^* (Sp(n), \R) = \Lambda (x_3,x_7, \cdots, x_{4n-1})$.\\
D. $H^* (SO(2n), \R) = \Lambda (x_3,x_7, \cdots, x_{4n-5}, y_{2n-1})$.\\
E6. $H^*(E_6, \R) = \Lambda (x_3, x_9, x_{11}, x_{15}, x_{17}, x_{23})$.\\
E7. $H^*(E_7, \R) = \Lambda (x_3, x_{11}, x_{15}, x_{19}, x_{23}, x_{27}, x_{35})$.\\
E8. $H^*(E_8, \R) = \Lambda (x_3, x_{15}, x_{23}, x_{27}, x_{35},  x_{39},  x_{47}, x_{59})$.\\
F. $H^*(F_4, \R) = \Lambda (x_3, x_{11}, x_{15}, x_{23})$.\\
G. $H^*(G_2, \R) = \Lambda (x_3, x_{11})$.

5.  If $\g$ is a complex simple Lie algebra,  and $\g_0$  is its  real form, then
any  $Ad(\g_0)$-invariant form $\phi^l$ on $\g_0$ extends  to an $Ad (\g)$-invariant  invariant form on $\g$.
\end{remark}

\section{Examples  of manifolds  provided with a  3-form  $\om ^3$ of  type $\om_\g$}

In this section we assume that $\g$ is  a  complex simple Lie algebra  of dimension $n\ge 8$  or a real form  of such a
Lie  algebra.  We introduce
the notion of a differential 3-form $\om ^3$ of type $\om_\g$, see Definition \ref{car1}. We   show some examples of  manifolds  provided
with a differential 3-form $\om ^3$ of type $\om_\g$, and  we discuss     some   possible  ways to construct  such manifolds  (Examples (\ref{ex1}) - (\ref{ex4})).

\begin{definition}\label{car1} A differential 3-form $\om^3$ on a manifold $M^n$  is  said  to be {\it  of  type   $\om_\g$}, if at every point $x\in M^n$ the 3-form $\om^3(x)$ is  equivalent to   
the Cartan form $\om_\g$ on $\g$, i.e. any linear isomorphism from $\g$ to $T_xM^n$  sends $\om^3(x)$ to a 3-form which belongs to 
the   $GL(\g)$-orbit  of $\om_\g$.  If $\g$ is a complex Lie algebra  we require that $M^n$ possesses  a volume form as well as an almost complex
structure $J$ and  the mentioned above linear  isomorphism    commutes with the (almost)  complex structures on $T_xM^n$ and on $\g$, moreover this
linear isomorphism preserves  a fixed volume form on $\g$.
\end{definition}

A differential 3-form $\om^3$ in Definition \ref{car1} is also called {\it a (simple, compact) Cartan 3-form}, if
no misunderstanding occurs.

 By Theorem \ref{dyn1}.2  the existence of a simple  Cartan 3-form on $M^n$ is equivalent to the existence
of an $Aut (\g)$-structure on $M^n$, if $\g$ is a real form of  a complex simple Lie algebra. Below we show   examples  and a possible construction 
of  manifolds $M^n$ admitting a  3-form of type  $\om_\g$, where  $n\ge 8$. (Note that any 3-manifold is parallelizable, so  it admits a 3-form  of type $\om_{\su(2)}$.)

\begin{example}\label{ex1} Any parallelizable  manifold  $M^n$   is provided  with a simple  Cartan 3-form, if
$n = \dim _\R \g$. For example  any (simple) Lie group $G$  is parallelizable,  the  manifold $S^{n_1} \times \cdots \times S^{n_r}$, $r\ge 2$, $\sum n _i = n$,  is parallelizable,
if  at least one  of the $n_i$ is odd \cite{Kervaire1956}.
\end{example}

\begin{example}\label{hom}   Assume that $\g$ is a real form
of a simple complex Lie algebra. A homogeneous   space $K/H$ admits   a  Cartan 3-form of type $\om_\g$  if and only if  the isotropy action  of $H$  is contained in the group  $Ad(\g) \subset \ssl (V) $, where $V = T_e K/H$ and $\dim V = \dim \g$.
Let $G$ be a Lie group with Lie algebra $\g$, and $\rho: H \to G$ - an embedding. We define $K:= H \times G$, and  let  $i: H \to K = H \times G$   have  the form $i = (Id \times \rho)$.   Note that  the isotropy action of $H$ on $V$ is contained in the group $Ad(\g)$, and this action  is reducible unless $\dim  H = \dim G$.  
For  given groups $H, G$ the homogeneous spaces $(H \times G )/H$ may have  infinitely  many distinct  homotopy types  depending on an embedding $\rho : H \to G$, see  \cite{KS1991} for the case $G = SU(3)$ and $H =  U(1)$.
\end{example}

\begin{example}\label{coh1}   Assume that a manifold $M^n$ is equipped with a differential 3-form $\om^3$ of type $\om_\g$ and a  Lie group $K$  acts on $M$  with cohomogeneity 1 preserving  the form $\om ^3$.
Assume that $\g$ is a real form of a simple complex Lie algebra. Then the principal $K$-orbit on $M^n$ has the form
$K/H$, where  the Lie algebra $\h$ of $H$  is also a sub-algebra of $\g$ such that the  induced adjoint action of $\h$ on $\g$  has a trivial   component of  dimension 1. Conversely, if $H$ is a subgroup of $K$ such that the sum of  the isotropy
action of $H$ on $T_e(K/H)$ with the trivial action of $H$ on $\R$ is  equivalent to the adjoint  action of $H$ on $\g$ via some embedding $\rho: \h \to \g$, then $(K/H)\times \R$ admits a differential 3-form
of type $\om_\g$.  In particular,   the  direct product $(G/S^1)\times  S^1$ admits
a 3-form of type  $\om_\g$, if $G$ is a Lie group with Lie algebra $\g$.  For a given $G$ these spaces may have infinitely  many distinct homotopy types, see \cite{KS1991}. Another example of a  compact cohomogeneity 1 space is manifold $(SO(5)/SU(2)) \times S^1$ admitting  a differential 3-form of type $\om_{\su(3)}$  as well as a differential 3-form
of type $\om_{\su(1,2)}$, since the sum of the isotropy representation of $SU(2)$ on $T_e (SO(5)/SU(2))$ with the trivial representation of $SU(2)$ on $\R$ is equivalent to
the adjoint action of $SU(2)$  on $\su(3)$ as well as to the one  on $\su(1,2)$.
\end{example}




\begin{example}\label{ex4}  Let $\g$ be a compact  simple Lie subalgebra of a compact Lie algebra $\bar \g$, and $\bar G$ -  a compact Lie group
with Lie  algebra  $\bar \g$. Let $\dim \g = k\ge 8$. Denote by $\Dd_k(\g\subset \bar\g)$ the collection  of  oriented $k$-planes $V^k$   on $\bar \g$ such that
$(\om_{\bar \g})_{| V^k}$ is equivalent to $\om_\g$. Let $\hat \Dd_k(\g\subset \bar\g)$ be the subset of the Grasmannian bundle $Gr_k(\bar G)$ of oriented $k$-planes tangent to $\bar G$ such that elements  of $\hat \Dd_k(\g\subset \bar\g)$ are obtained from $\Dd_k(\g\subset \bar\g)$ by translations composed from  left and right multiplications on $\bar G$. Let $M^k$ be a   submanifold of
$\bar G$  such that $TM^k \subset \hat\Dd_k(\g\subset\bar \g)$; we call such  $M^k$ {\it an  integral  submanifold of  $\hat\Dd_k(\g\subset\bar \g)$}. Then $M^k$
admits a closed 3-form of type $\om_\g$, which is the restriction of the Cartan 3-form on $\bar G$.  It is an interesting question,  if $\Dd_k(\g \subset \bar \g)$  has an integral submanifold   not locally isomorphic to the connected Lie subgroup $G$ having Lie algebra $\g$ in $\bar G$. Note that $\Dd_k(\g\subset\bar\g)$  contains  the subset $\hat \Dd(\g)$
which is obtained from subspace $\g\subset \bar \g$ by translations composed from  left and right multiplications on $\bar G$.
Liu proved that  any integral submanifold of $\hat\Dd(\g)$ is a totally geodesic submanifold in $\bar G$ \cite{Liu1995}. His result
generalized  a previous  result by Ohnita and  Tasaki \cite{OT1986}.  
\end{example}

\section{Algebraic types of manifolds provided with a 3-form $\om^3$ of type $\om_\g$}

In this section   we assume that  $G$ is a  compact simple Lie  group with Lie algebra $\g$ of dimension at least 8. We  recall the notion of  the intrinsic  torsion of a $G$-structure on a manifold $M^n$, and the notion of the  algebraic type of a $G$-structure, specializing  to the case
$G = Aut(\g) \subset SO(\g)$  (Definitions \ref{int},  \ref{aty}). 
We   prove some structure theorems on  $\g$-submodules of the $\g$-module  $\g \otimes \g^\perp$ with focus on those
intrinsic torsions whose affine  torsion is skew-symmetric  (Remark \ref{ten}, Remark \ref{skew}, Proposition \ref{main1}, Theorem \ref{main2}  and Corollary \ref{skew2}).
We  prove that any complete torsion-free $Aut(\g)$-manifold $M^n$ is either  a quotient of $\R^n$
or irreducible and locally symmetric of type I or IV  (Theorem \ref{tor}).
\medskip

Suppose that $M^n$ is a manifold  equipped with a  differential 3-form $\om^3$ of type $\om_\g$. By Theorem \ref{dyn1}  $M^n$ is equipped with an  $Aut(\g)$-structure, and
hence with  a Riemannian metric $K_\g$ which is associated with the negative of the Killing metric $K $ on $\g$, see  Remark \ref{fried}.1 below. 
Let $\hat\nabla$ be an $Aut(\g)$-connection of the $Aut(\g)$-structure on $M^n$, and $\nabla ^{LV}$ the Levi-Civita connection of the Riemannian metric
$K_\g$. 
Then $\eta: = \hat\nabla - \nabla ^{LC}$ is a tensor taking values in $T^*M^n \otimes \so(TM^n,K_\g)$, which  is
called {\it the  torsion tensor of $\hat \nabla$}. 
Using the isomorphism  $\so (\g) = \Lambda ^2 (\g)$
we get  the following $Ad(\g)$-equivariant decomposition
\begin{equation}
\so (\g)  = \g \oplus \g ^\perp, \label{decso}
\end{equation}
where 
\begin{equation}
 \g ^{\perp} : = \ker \delta_\g : \Lambda ^2 (\g) \to \g , \: v\wedge w \stackrel{\delta_\g}{\mapsto} [v, w].
 \label{lbr}
\end{equation}

\begin{remark}\label{del1}  Let us define a linear operator
\begin{equation}
d_\g: \Lambda ^1 (\g) \to \Lambda ^{2} (\g), \, \la d_\g(v), x\wedge y\ra : =\la v, [x, y]\ra. 
\end{equation}
Clearly,  $\delta _\g$ is the  adjoint  of $ d_\g$ with respect to the Killing metric $\la, \ra$.  
\end{remark}

Denote  by $\eta ^\g$ the component of the torsion tensor $\eta$ in $\g \otimes \g \subset \g\otimes \so(\g)$.   

\begin{definition}\label{int}\cite[\S 1]{Salamon2001} {\it The intrinsic torsion}
of an $Aut(\g)$-structure on $M^n$ is  defined by 
\begin{equation}
\xi: = \eta -  \eta ^\g  \in \g \otimes  \g ^\perp.
\label{tor1}
\end{equation}
\end{definition}

Since any  $Aut(\g)$-connection  on $(M^n, \om^3)$ is obtained from $\hat \nabla $ by adding a tensor taking value in $\g \otimes \g$,
the  intrinsic torsion is defined uniquely on $(M^n, \om_\g)$.

\begin{definition}\label{aty}  Let $W$ be a $\g$-submodule  in $\g \otimes \g ^\perp$.
A manifold $M^n$ provided with a 3-form $\om ^3$ of type $\om _\g$ is called of {\it  algebraic type of $W$}, if its intrinsic torsion $\xi$ takes values
in $W$.
\end{definition}

\begin{remark}\label{fried}  1. Given  a differential  3-form  $\om^3$   of type $\om_\g$ on $M^n$, we  define  the associated metric $K_\g$  by specifying a linear
isomorphism   $I_x:T_xM^n \to \g$ sending $\om_\g$ to $\om^3_x$.  Then $I_x^* (-K)$  is the required  Riemannian metric $K_\g (x)$. By Theorem \ref{dyn1}   the obtained  metric does not  depend on the choice of $I_x$.

2. (cf.  \cite[Lemma 1.2]{Salamon2001}) The above definition of algebraic types of  $Aut(\g)$-structures   is a specialization 
of   a definition of  algebraic types of  $G$-structures on manifolds $M^n$, where $G \subset SO(n)$. 
This scheme has been suggested first by Gray and  Hervella  for  almost Hermitian manifolds \cite{GH1980}.
In fact  they considered the case of a group $G$ being the stabilizer of   a tensor $T\in V$, where $V$ is a tensor space over $\R^n$ with induced action of $G$, and they looked at  the $G$-type of the  tensor $\nabla^{LC} T$.  Since $\hat\nabla T = 0$, where as before $\hat\nabla$  is a $G$-connection, we get
\begin{equation}
 \nabla ^{LC}T = \eta (T) =  \xi (T) = \sum_i e_i \otimes( \rho_* (\xi( e_i))  T \in \g \otimes  V,
 \label{rel1}
 \end{equation}
where $\rho_*: \so (n) \to \so (V)$ is the differential of the induced embedding $\rho : SO(n) \to SO(V)$, and $(e_i)$
is  an orthonormal basis in $\g$, so $\xi( e_i)$ is  a contraction of $\xi $ with $e_i$, which  takes values in $\g^\perp \subset \so (\g)$.  It follows from (\ref{rel1}) that    if $G\subset SO(n)$ is the stabilizer
group of a tensor $T$  on $\R^n$, then 
the algebraic $G$-type of $\nabla^{LC} T$ defines  the algebraic  $G$-type of the intrinsic torsion $\xi$  of a $G$-structure and vice versa.\\
3. In the case $T$ is  a 3-form of type $\om_\g\in \Lambda ^3 (\g)$ the formula (\ref{rel1}) has the following simple expression (see  the proof of Lemma \ref{dgs}.2 below)
\begin{equation}
\rho_* (\tau)(\om_\g) = d_\g \tau, \label{dg2}
\end{equation}
for $\tau \in \so (\g) = \Lambda ^2 (\g)$.
\end{remark}

\begin{remark}\label{ten} 1. For  a compact simple Lie algebra $\g$  we have computed the  decomposition of $\g^\perp\otimes \C$ into irreducible components using table 1 and table 5 in \cite{OV1990}. We put the result in Table 1.
We keep  notations in \cite{OV1990} with $k\pi_i$   denoting the  irreducible representation of the  highest weight  $(0, \cdots, k_{(i)},\cdots,  0)$ with respect to a  basis of  simple roots  of $\g$, and $R(\Lambda)$ denotes the  irreducible  representation with  the highest weight  $\Lambda$.
\medskip

{\bf Table 1}:  {\it Decomposition  of $\g^\perp\otimes  \C$}.
We denote by $\delta(\g)$ the  highest weight of the adjoint representation of $\g\otimes \C$. We note that $\so(6) \cong \su(4)$.
\medskip

\begin{tabular}{l|l|l|l}
$\g$ & $\dim \g $ & $\delta(\g)$ &  $\g^\perp\otimes \C$\\  
$\su (n+1)$ & $n ^2 + 2n$ &   $\pi_1 + \pi_l$& $  R(2\pi_1 +\pi_{n-1}) + R(\pi_2 + 2\pi _n)$\\
$\so(2n+1)$ & $ n(2n+1)$ & $\pi_2$ &  $ R(\pi_3 + \pi _1)$\\
$\ssp(n)$ & $n(2n+1)$ & $2\pi_1$ & $R(2\pi_1 + \pi_2) $\\
$\so(2n), \, n\ge 4$  & $n(2n-1) $ & $ \pi_2$ & $R(\pi_3 + \pi _1)$\\
$E_6$  & $ 78$  &$\pi_6$ &  $ R(\pi_3)$ \\
$E_7 $  & $133 $ & $\pi_6$ & $ R (\pi_5) $\\
$E_8 $  & $ 248 $ & $\pi_1$ &  $ R(\pi_2)$ \\
$F_4$ & $52 $ &  $\pi_4$ & $ R (\pi_3) $\\
$G_2$ &  $ 14 $ & $\pi_2$ & $R (3\pi_1)$  
\end{tabular}

\medskip

Thus    except $\g = \su (n+1)$,   in all other cases $\g^\perp\otimes \C$ is irreducible. (I thank Dmitri Panyushev, who
informed me of this observation).
In case $\g = \su(n+1)$, since $R(2\pi_1 +\pi_{n-1})$ and $ R(\pi_2 + 2\pi _n)$  are complex conjugate, the  real  $\su(n+1)$-module $ \g^\perp$
is also irreducible. We refer the reader to \cite[ \S 8]{Onishik2004}  for a comprehensive  exposition of the theory of  real representations of real semisimple Lie algebras.

2. To find  a decomposition  of the tensor  product $(\g \otimes \g^\perp)\otimes \C$ into irreducible  components
for a given simple Lie algebra $\g$ we  could use  available     software program packages (GAP or LiE or something else).  Note that a general  formula of this decomposition  for a compact simple Lie algebra $\g$ in any  infinite series $A_n, B_n, C_n, D_n$ is not  known.  Below we will find some important  $\g$-submodules  in  the  $\g$-module $\g \otimes \g^\perp$.
First we note that 
the  irreducible component with the largest dimension  of the tensor product $R(\delta(g)) \otimes R(\Lambda)$   has the highest weight $\delta(\g) + \Lambda$  \cite[Theorem 3.1]{Dynkin1952}.  (In \cite[Theorem 3.1]{Dynkin1952}
Dynkin gave   a  simple method to  find  some less  obvious irreducible components of  the tensor product of two irreducible complex representations.)  Let $\Lambda ^\perp$ denote the highest weight of  the irreducible   representation $\g^\perp\otimes \C$ if $\g \not = \su (n+1)$, and   let $\Lambda ^\perp$  denote the weight $(2\pi_1  + \pi _{n-1})$, if
$\g = \su (n+1)$. It follows that $\g \otimes \g^\perp$ contains an irreducible component  $R_{\max}: =Re (R ( \delta   + \Lambda ^\perp))$ if  $\g \not = \su (n+1)$, and $R_{\max} : = R(\delta + \Lambda ^\perp)\otimes \R$ if $\g = \su (n+1)$.   Here we denote by $Re (V)$ a real form  of a complex vector space $V$, and by $V\otimes \R$ - its realification.
\end{remark}

\medskip

Let us  consider the following $Ad(\g)$-equivariant linear
maps
\begin{equation}
D_+ : \g \otimes \g^\perp \to \Lambda ^4 (\g): \, v\otimes \tau \mapsto \sum v \wedge \rho_*( \tau)( \om _\g),\label{do1} 
\end{equation}
\begin{equation}
D_- : \g \otimes \g^\perp \to \Lambda ^2 (\g): \: v\otimes \tau \mapsto \sum v \rfloor  \rho_*(\tau)( \om _\g),\label{do2} 
\end{equation}
where $\rho_*(\tau)$  acts on the space $\Lambda ^3 (\g)$ as we have explained in Remark  \ref{fried}.

\begin{proposition}\label{main1} 1. A 3-form $\om ^3$ of  type $\om_\g$ is closed, if and only if its intrinsic torsion $\xi$ satisfies $\xi(x) \in \ker D_+(x)$  for all $x \in M^n$.\\
2. A  3-form $\om ^3$ of type $\om_\g$ is co-closed, if and only if its intrinsic torsion $\xi$ satisfies  $\xi(x) \in \ker D_-(x)$  for all $x \in  M^n$.\\
3. The  irreducible component $R_{\max}$ belongs to $\ker( D_-)$.\\
4. The irreducible component $R_{\max}$ belongs to $\ker (D_+)$ if and only if $\g = \su (3)$. 
\end{proposition}

\begin{proof}   Let us explain the meaning of the operators  $D_\pm$. We denote by $(e_i)$  an orthonormal frame  in $\g$ and $(e^i)$ its dual frame. 
Using the  following well-known identities  for a  differential form $\rho$ on  a Riemannian manifold $M^n$
\begin{eqnarray}
d\rho(x) = \sum_i e^i\wedge \nabla ^{LV}_{e_i} \rho (x) ,\label{kn1}\\
 d^* \rho (x) =   \sum _ie_i \rfloor \nabla ^{LC} _{e_i} \rho (x) ,\label{kn2}
 \end{eqnarray}
taking into account (\ref{rel1}), we get the first and second assertions of Proposition \ref{main1}.

 Since $\g$ and $\g^\perp$ are the only irreducible components of $\Lambda ^2 (\g)$, the image  $Im \, (D_-) \subset \Lambda ^2 (\g)$ does not contain  any component equivalent to $R_{\max}$.  It follows  that   $R _{\max}\subset \ker (D_-)$.  This proves  the  third assertion of Proposition \ref{main1}. 

4. We will show that $R_{\max} \subset\ker (D_+)$ if and only if $\g = \su(3)$.  First we recall that the   $Ad (\g)$-equivariant  differential  $d_\g:  \Lambda ^k (\g) \to \Lambda ^{k+1} (\g)$ is  defined as follows. For $k = 1$
we define $d_\g$ as in Remark \ref{del1}.  For $k \ge 2$ we define $d_\g$ inductively by
\begin{equation}
d_\g (\alpha \wedge \phi) : =  d_\g (\alpha) \wedge \phi + (-1) ^{deg\, \alpha}\alpha \wedge d_\g (\phi).\label{dg1}
\end{equation}
 Next  we define  linear operators $ \Pi_{1,3}^+: \g \otimes \Lambda ^3 (\g)  \to  \Lambda ^4 (\g)$ and
 $\Pi^{-}_{1,3}: \g \otimes \Lambda ^3 (\g)  \to  \Lambda ^2(\g)$  by
$$ \Pi_{1,3}^+( v\otimes \theta ):=  v\wedge \theta, \: \Pi_{1,3}^-( v\otimes \theta ):=  v\rfloor \theta.$$

\begin{lemma} \label{dgs} 1. The operator $d_\g$ can be  expressed by
\begin{equation}
d_\g ( \phi) = \sum _i(e_i \rfloor \om_\g) \wedge (e_i \rfloor \phi), \label{dg2b}
\end{equation}
where $(e_i)$ is an  orthonormal basis in $\g$.

2. Let $\tau \in \g^\perp$. Then 
\begin{equation}
D_+ ( v \otimes \tau) = \Pi_{1,3}^+ ( v \otimes  d_\g ( \tau)), \: D_- ( v \otimes \tau) = \Pi_{1,3}^- ( v \otimes  d_\g ( \tau)).\label{dpl}
\end{equation}
\end{lemma}

\begin{proof} 1. We  check easily  the validity  of (\ref{dg2}) for $\phi\in \Lambda ^1 (\g)$.  Denote the RHS of (\ref{dg2}) by
$\sigma _+ (\phi)$.
We   observe that $\sigma _+$ is a differential, i.e. $ \sigma _+ (\alpha \wedge \beta ) = \sigma _+ (\alpha) \wedge
\beta +(-1)^{deg\, \alpha} \alpha \wedge \sigma_+(\beta)$. Hence $d_\g = \sigma _+$. This proves the first assertion of Lemma \ref{dgs}.

2. Note that $D_+$ is the restriction of  the linear operator also denoted by $D_+ : \g \otimes \Lambda ^2 (\g)$ defined by the same formula
in (\ref{do1}). Thus it suffices to check the validity of (\ref{dpl})  for basis elements $e_i \otimes (e_j \wedge e_k) \in \g \otimes \Lambda ^2 (\g)$.
Using (\ref{do1}) we get
$$ D_+ (e_i \otimes (e_j \wedge e_k)) = e_i \wedge [e_j \wedge (e_k \rfloor \om_\g) - e_k \wedge (e_j \rfloor \om_\g)].$$
Using (\ref{dg2}) we get
$$\Pi_{1,3}^+ (e_i \otimes d_\g ( e_j \wedge e_k)) =  e_i  \wedge [(e_k \rfloor \om_\g) \wedge e_j  - e_k \wedge (e_j \rfloor \om_\g)].$$
Comparing the above formulas we get 
$$D_+ (v\otimes \tau) = \Pi^+_{1,3} (v \otimes \tau).$$
In the same way we get $D_-(v\otimes \tau) = \Pi^{-}_{1,3} (v\otimes \tau)$. This completes  the proof of the second assertion of Lemma \ref{dgs}.
\end{proof}

Let us continue the proof of Proposition \ref{main1}, keeping the notations in Table 1. Let $\alpha_1$ be  the  simple root of $\g \otimes \C$
such that $(\delta, \alpha_1) \not = 0$ and $v_{\delta} \wedge v_{\delta -\alpha_1}$ is  (one of) the highest  vector of $\g^\perp \otimes \C$, where $v_\delta\in \g\otimes \C$ (resp. to $v_{\delta -\alpha_1}\in \g\otimes \C$) is the  root vector  corresponding to $\delta$ (resp.
$\delta -\alpha_1$),   see e.g. \cite[Theorem 3.1]{Dynkin1952}. Note that the component $R_{\max}\otimes \C$ has the highest vector
$v_\delta \otimes (v_{\delta} \wedge v_{\delta -\alpha_1})$.  To prove the last assertion of  Proposition \ref{main1}
it suffices to show that  $D_+ (e_\delta \otimes (e_{\delta} \wedge e_{\delta -\alpha_1}))=0$
if and only if $\g = \su(3)$, where $D_+$, $\Pi_{1,3}^+$ and $d_\g$ extend $\C$-linearly to
$(\g\otimes \g^\perp)\otimes \C$ and $\Lambda ^k(\g \otimes \C)$.
A direct computation using (\ref{dpl}) shows that 
$$D_+ (e_\delta \otimes (e_{\delta} \wedge e_{\delta -\alpha_1}))= e_\sigma \wedge \sum_{\alpha +\beta = \sigma} c^{\sigma} _{\alpha,\beta} e_\alpha \wedge e_\beta \wedge e_{\sigma -\alpha_1},$$
where $d_{\g}e_\sigma = \sum_{\alpha +\beta = \sigma} c^{\sigma} _{\alpha,\beta} e_\alpha \wedge e_\beta$. 
Using the table of simple roots  of simple Lie algebras \cite[Table 1]{OV1990} and the above formula, we conclude  that $D_+ (e_\delta \otimes (e_{\delta} \wedge e_{\delta -\alpha_1}))= 0$, if and only if $\g = \su(3)$.  This completes the proof
of Proposition \ref{main1}.
\end{proof}
\medskip

We introduce new notations by looking at the following orthogonal decompositions
$$W_{har} : = \ker (D_+) \cap \ker (D_-)\subset \g \otimes \g^\perp.$$
$$\ker (D_+) := W_{har} \oplus W_{har}  ^\perp(D_+).$$
$$\ker (D_-) := W_{har} \oplus W_{har} ^\perp ( D_-).$$
Denote by $\delta _\g$ the adjoint of $d_\g$ with respect to the minus Killing metric on $\g$.  Note that  we have the following  orthogonal decomposition
\begin{equation}
\Lambda ^3 (\g) = \la \om_\g \ra _\R \oplus  \Lambda ^3_d(\g) \oplus \Lambda ^3 _\delta (\g), \label{dec3}
\end{equation}
where $\Lambda ^3_\delta (\g) := \delta_\g (\Lambda ^ 4(\g))$  and $\Lambda ^3_d(\g) = d_\g ( \g ^\perp)$. 

\begin{theorem}\label{main2} 1. We have the following decomposition  
$$ \g \otimes \g^\perp = W_{har} \oplus_i  V_i, $$
where $V_i$  is one of  irreducible modules contained in $\Lambda ^2 (\g) \oplus \Lambda ^4 (\g)$.\\
2.  $W_{har}^\perp(D_+)$ contains  a $\g$-module which is
isomorphic to $\g^\perp \subset \Lambda ^2 (\g)$.\\
3.  $W_{har}^\perp (D_-)$ contains  a $\g$-module which is isomorphic to $\Lambda ^3_\delta (\g)$.\\ 
4.  The image of $D_+ $ contains  the module $\om _\g \wedge \Lambda ^ 1(\g)$.\\
5.  $D_-$ is surjective.
\end{theorem}

\begin{proof} 1. We have a decomposition $\g\otimes \g^\perp = \ker (D_+ ) \oplus \ker( D_+ )^\perp$.  Let us  consider the decomposition $\ker (D_+) = W_{har} \oplus  W_{har}^{\perp}(D_+)$.
Since the restriction of $D_-$ to $W_{har}^\perp (D_+)$ is injective, $W_{har}^\perp (D_+)$  is isomorphic to  a $\g$-module in $\Lambda ^2 (\g)$. 
Since $(\ker D_+) ^\perp$ is isomorphic to a $\g$-submodule in $\Lambda ^4(\g)$, we  get the first assertion of Theorem \ref{main2} immediately.

2. We define  an $Ad(\g)$-equivariant  linear map
\begin{equation}
\Theta : \Lambda ^3 (\g) \to \g \otimes \g^\perp , \:  T \mapsto \sum_i e_i \otimes \Pi_{\g^\perp} ( e_i \rfloor  T),
\end{equation}
where $\Pi _{\g ^{\perp}}$ is the orthogonal projection onto   the  subspace $\g^\perp \subset \Lambda ^2 (\g)$
and $(e_i)$ is an orthonormal basis in $\g$.

\begin{remark}\label{skew} A result by Friedrich  \cite[Theorem 3.1]{Friedrich2002} implies that   the image $\Theta(\Lambda^3(\g))$ consists of all  intrinsic torsions $\xi ^A$ of  $Aut(\g)$-connections $A$ whose affine torsion  $T^A: =  \nabla ^A _X Y - \nabla ^A _Y X - [X, Y ]$ is skew-symmetric, i.e. $\la T^A (X, Y), Z \ra = -
\la T^A(X, Z), Y\ra $.
\end{remark}



\begin{lemma}\label{dth}  For  any $\theta \in \Lambda ^3 (\g)$  we have \\
1. $(D_+ )\circ \Theta (\theta) = -3 d_\g \theta$.\\
2. $(D_-) \circ \Theta  (\theta) = -\delta _\g (\theta)$.
\end{lemma}

\begin{proof} 1. Let $\theta\in \Lambda ^ 3 (\g)$.  Using (\ref{dpl}) we  get
\begin{equation}
 D_+\circ  \Theta (\theta) =  \sum _i ( e_i \wedge d_\g ( e_i \rfloor \theta)),\label{dth1}
 \end{equation}
where $(e_i)$ is an orthonormal basis in $\g$. Applying the Cartan formula  $d_\g(e_i\rfloor \phi) = -e_i\rfloor d_\g\phi + ad_{e_i} \phi$  to the RHS of (\ref{dth1}) we get
\begin{equation}
D_+\circ \Theta (\theta)  =  \sum _i e_i \wedge[ ( - e_i  \rfloor  d_\g \theta)  + ad _{e_i}  (\theta)]  = - 4 d_\g \theta + \sum_i e_i \wedge ad_{e_i} ( \theta).\label{dth2}
\end{equation}
Set $\tilde d_\g: = \sum _i e_i \wedge ad_{e_i}$. Note that  the  operator $\tilde d_\g :\Lambda ^k(\g) \to \Lambda ^{k+1} (\g)$ is a differential,
i.e. 
$\tilde d_\g ( \alpha \wedge \beta) =  \tilde d_\g (\alpha) \wedge \beta +(-1) ^{deg\, \alpha} \alpha \wedge \tilde
d_\g \beta$.   Furthermore we check easily for $\alpha \in \Lambda ^1 (\g)$ that $\tilde  d_\g (\alpha ) = d_\g(\alpha)$.
It follows that $d_\g = \tilde d_\g$. Using this we get the first assertion of Lemma \ref{dth} immediately from (\ref{dth2}).

2. Let $\theta \in \Lambda ^ 3 (\g)$.  Using (\ref{dpl}) and the Cartan formula $ad_{e_i} \theta = d_\g (e_i \rfloor \theta) + e_i \rfloor d_\g \theta$, we  get
\begin{equation}
 D_- \circ\Theta (\theta) =  \sum _i ( e_i \rfloor  d_\g ( e_i \rfloor \theta)) =  \sum _i (e_i \rfloor ad_{e_i} \theta).\label{dth3}
 \end{equation}
 
On the other hand we have
\begin{equation}
\delta_\g (v_1 \wedge    \cdots  \wedge v_k ) = \sum _{i<j} (-1) ^{i +j +1} [ v_i, v_j] \wedge v_1 \wedge \cdots _{\hat i} \cdots _{\hat j}\cdots \wedge v_k. \label{delg}
\end{equation}

Now compare (\ref{dth3}) with (\ref{delg}) we get  the second assertion  of Lemma \ref{dth} immediately.
\end{proof}

\begin{lemma}\label{kth}  $\ker \Theta = \la \om _\g \ra _\R.$
\end{lemma}

\begin{proof}  Clearly $\phi \in \ker \Theta$ if and only
if $v_i \rfloor  \phi = d_\g w_i$ for all $v_i$ and  some $w_i  \in \g$ depending on $v_i$.
In particular $\om _\g \in \ker \Theta$. 

Now let $\phi \in \ker \Theta$. We write
$$ \phi = \phi_{harm} + \phi_d  + \phi _\delta$$
corresponding to the decomposition (\ref{dec3}). To complete the proof of Lemma \ref{kth} it suffices to show that
$\phi_d = 0 = \phi_\delta$.  Using Lemma \ref{dth} we get
$$ (D_+)\circ \Theta (\phi) = -3d_\g \phi_\delta, $$
$$(D_-)\circ \Theta (\phi)= -\delta_\g \phi _d.$$
Hence we get
$$d_\g \phi_\delta = 0 = \delta_\g \phi_d.$$
Since $(\ker d_\g) _{| \Lambda _d^3 (\g)} = 0$ and
$(\ker \delta _\g) _{|\Lambda ^3 _{\delta}(\g)} = 0$, we conclude that
$\phi_d = 0 = \phi_\delta$. This completes the proof of Lemma \ref{kth}.
\end{proof}

 Let us continue the proof of  Theorem \ref{main2}.2. It follows from Lemma \ref{dth}
\begin{equation}
 D_+( \Theta (d_\g (\g^\perp)) )= 0,  \label{d+1}
 \end{equation}

 \begin{equation}
  D_- (\Theta (d_\g (\tau))  =-\delta_\g d_\g (\tau) \text{ for } \tau \in \g^{\perp}.\label{mid}
 \end{equation}
It follows that $\Theta (d_\g (\g^\perp)) \subset W_{har} ^\perp (D_+)$. Taking into account Lemma \ref{kth} this proves the second assertion of  Theorem \ref{main2} immediately.

3. By Lemma \ref{kth} $\ker \Theta _{|\Lambda ^3 _\delta (\g)}=0$.  Lemma \ref{dth} implies that  $\ker (D_-)$ contains  a $\g$-submodule $\Theta (\Lambda ^3 _\delta (\g))$,  and moreover  the kernel of the restriction of $D_+$ to $\Theta
(\Lambda ^3 _\delta (\g))$ is zero. 
This proves the third assertion of Theorem \ref{main2}.

4. The fourth assertion of Theorem \ref{main2} follows by comparing  (\ref{do1}) with the following formula 
$$ \sum _i  e_i \wedge  \rho_*(e_i \wedge X)(\theta) = 3X\wedge \theta,$$
where $(e_i)$ is an orthonormal basis in $\g$, $X \in \g$  and $\theta \in \Lambda ^3(\g)$, in particular this formula holds for $\theta = \om_\g$.  

5. Using  (\ref{mid}) we conclude that the image of $D_-$ contains   $\g ^\perp$. A direct computation yields the following identity  for any $X \in \g$
$$\sum _i e_i \rfloor  \rho_*(e_i \wedge X)(\om _\g) = -2 X\rfloor \om _\g = -2 d _\g X,$$
where $(e_i)$ is an orthonormal basis in $\g$. It follows that  the image of $D_-$ contains the  irreducible component $d_\g \g \subset \Lambda ^2(\g)$.
Since $\Lambda ^2 (\g) = d_\g (\g) \oplus \g^\perp$, this completes the proof of Theorem \ref{main2}.
\end{proof}

From Remark \ref{skew}, Lemma \ref{kth} and Lemma \ref{dth} we get immediately

\begin{corollary}\label{skew2} The space of $Aut(\g)$-connections with skew-symmetric affine torsion on a manifold $M^n$   provided with a 3-form $\om^3$ of type $\om_\g$ is a direct sum  of two $\g$-modules, one of them consists of  those
connections for which $d\om^3 = 0$ and the other one consists of those  connections for which $d^* \om^3 = 0$.
\end{corollary}

It follows that a manifold $M^n$ admitting a harmonic form $\om^3$ of type $\om_\g$ and having an $Aut(\g)$-connection
with skew-symmetric torsion is  in fact torsion-free.

\begin{theorem}\label{tor}  Let $M^n$ be a complete torsion-free $Aut(\g)$-manifold.  Then $M^n$ is either flat, or
$M^n$ is irreducible  and locally symmetric  of type I or IV.
\end{theorem}

\begin{proof} Let $\h$ be  a Lie subalgebra in $\g$.
We write $\g = \h \oplus V$,  where $V$ is orthogonal complement to $\h$. Since $\g$ is simple, the adjoint representation  $ad_\g (\h)$ restricted to $V$ is nontrivial. Taking into account the de Rham decomposition theorem,
we conclude that $M^n$ cannot have a  holonomy group $H$  strictly smaller than $Ad(\g)$ unless $H = Id$. Hence 
$M^n$ must be either irreducible and locally symmetric, or flat. Taking into account Theorem \ref{3para} we conclude that,
if $M^n$ is locally symmetric,  then it is of type I or IV. This completes the proof of Theorem \ref{tor}.
\end{proof}

\begin{remark}\label{ad1} 1. A complete list of algebraic types of $PSU(3)$-structures is given in Witt's  Ph.D.
Thesis \cite{Witt2004}, but Witt studied  only $PSU(3)$-structures corresponding to   a harmonic form $\om_{\su(3)}$. In \cite{Puhle2010} Puhle  studied a large class of  algebraic types of $PSU(3)$-structures in   great detail.\\
2. We could extend many results in this section to  the case of simple noncompact Lie group, using the theory
of real representation of  semisimple Lie algebras \cite{Onishik2004}. 
\end{remark}


\section{Cohomology theories for  $Aut^+ (\g)$-manifolds  equipped with a quasi-closed form   of type $\phi_0^l$}

In this section we also assume that $\g$ is a compact simple Lie algebra of dimension $n\ge 8$.
Let $Aut ^+(\g) : = Aut (\g) \cap GL ^+ (\g)$ and $\phi^l_0$  an $Aut ^+(\g)$-invariant $l$-form on $\g$. We study  necessary and sufficient conditions for an orientable $Aut(\g)$-manifold  to admit a multi-symplectic form
$\phi^l$ of type $\phi^l_0$ satisfying $d\phi^l =  \theta \wedge \phi^l$  (Theorem \ref{main3} and Lemma \ref{exi2}). Such a differential form  $\phi^l$ will be called {\it a quasi-closed form}.
Furthermore, 
we  construct  cohomology groups for two differential complexes $(\Om^* _{\phi^l_\pm}(M^n), d_{\pm})$  on  an $Aut^+(\g)$-manifold $M^n$ provided with a quasi-closed form $\phi^l$ of type $\phi^l_0$, see Proposition \ref{closed}. 
We consider a  spectral  sequence relating these cohomologies with the deRham cohomologies of $M^n$  (Theorem \ref{spec2}). We compute these groups in the case of 8-manifolds admitting a harmonic 3-form of type $\om_{\su(3)}$ in Example \ref{csu3}. We  introduce the notion of $\phi^l_-$-harmonic forms  (Definition \ref{har})
and   show some relations between  $\phi^l_-$-harmonic forms and  the group $H^*_{\phi^l_-} (M ^n)$ (Proposition \ref{har2}
and Lemma \ref{har1}). 

\begin{theorem}\label{main3} Assume that $\phi^l_0 \not=0$.

1. Any  $Ad(\g)$-invariant form  $\phi^l_0$ on $\g$ is  multi-symplectic, if $\g$ is a simple Lie algebra over $\C$ or over $\R$.

2. Let $\g$ be a  classical compact simple Lie algebra. Then the algebra $\Lambda _{Aut ^+} (\g)$ of $Aut ^+ (\g)$-invariant  forms on $\g$ coincides with the algebra $\Lambda_{\g} (\g)$ of $Ad(\g)$-invariant
forms on $\g$ except the case  $\g  = su (n+1)$ where $ 4 $ divides $ n (n +3)$.

3. Let $\g  = su (n+1)$ such that $ 4 $ divides $ n (n +3)$. Then $\Lambda  _{Aut ^+}(\g) = \Lambda ( x_{4k -1},  x_{4l+1} x_{4m +1})$ where
$x_{4p\pm 1}$ are primitive  generators  of  $\Lambda _\g (\g)$.
\end{theorem}

\begin{proof}   Recall that a form $\phi^l_0\in \Lambda ^l(\g^*) \cong \Lambda^l(\g)$ is  multi-symplectic if and only  the map
$$L_{\phi^l_0} : \g \to \Lambda ^{l-1} (\g):  v \mapsto v\rfloor  \phi^l_0$$
is injective.  Since $\phi^l_0$ is $Ad(\g)$-invariant,  the kernel of $L_{\phi^l_0}$ can be either  $0$ or  $\g$.
Since $\phi^l_0\not =0$ there exists $v\in \g$ such that  $v \rfloor \phi^l_0 \not = 0$.  Hence  $\ker L_{\phi^l_0} =0$,  this proves
the first  assertion  of Theorem \ref{main3}.

Let  $\g  = \su (n +1)$. It is known that $Aut (\g)$ is generated by $Ad(\g)$ and  the complex conjugation $\sigma$
on $\su(n+1)$.
We can check easily that  $\sigma$  is orientation preserving if and only if $4$  divides $n(n+3)$, see also Remark
\ref{e6} below.
Clearly $\sigma$ acts on the primitive elements $x_{2i +1}$  by multiplying $x_{2i+1}$ with $\pm 1$. To find   exactly the sign
of this multiplication we note that
$$ (\sigma (\phi^l_0), v) = (\phi^l_0,\sigma (v))$$
for any $l$-vector $v\in \Lambda ^l (\su (n+1))$.  Now letting $v$ be  the unit $(2i+1)$-vector associated with the orthogonal complement
to $\su (i)$ in $\su (i +1) \subset \su (n+1)$, we conclude that 
$$\sigma (x_{2i +1}) = - x_{2i +1}, \text{ if } i = 2l,$$
$$\sigma (x_{2i +1}) =  x_{2i +1} \text{  if } i = 2l-1.$$
 This  proves  the third assertion  and  the part of  the second assertion  of Theorem \ref{main3} concerning  $\g = \su (n +1)$.

Let  $\g = \so (2n)$. It is known that $Aut (\g)$ is generated by $Ad(\g)$ and the element $ \sigma = Ad (diag (1, \cdots, 1, -1)) \in Ad (O (2n))$.  Clearly  $\sigma$  reserves the orientation.  Hence $Aut ^+ (\g ) = Ad (\g)$,
if $\g = \so (2n)$ and $n \not =  4$.   Combining with  Remark \ref{hit}.2  we  obtain easily the second assertion  of Theorem \ref{main3} for the case $ \g \not = \so (8)$.

Note that    the  group $Aut  (\so (8))$ is generated by $Ad(\so(8))$ and $\Sigma _3$,  see Remark \ref{hit}.2.
Since $\Sigma _3$ is generated by $\sigma _i$, $i =\overline{1, 3}$, which is conjugate by an element in $Aut(\so(8))$  to  the element $\sigma= Ad (diag (1, \cdots, 1, -1)) \in Ad (O (8))$, it follows that  $Aut ^+ (\so (8))$  acts on 
$\Lambda _{\so (8)} (\so (8))$ as identity. This completes the proof of Theorem \ref{main3}.
\end{proof}

\begin{remark}\label{e6}  It has been observed by Todor Milev that an outer automorphism $\sigma$ of a simple
compact Lie algebra $\g$ is orientation preserving, if and only if the action of $\sigma$  on the Dynkin diagram
$D(\g)$ is a composition of even number of permutations.  To prove this statement for  classical compact Lie algebras he used an argument similar  to our argument  above. For the case $E_6$  he proved this assertion with a help
of a  computer program written by himself.  In particular $Aut ^+ (E_6) = Aut (E_6)$. We conjecture that   $\Lambda _{Aut ^+} (E_6) = \Lambda (x_3,  x_{11}, x_{15}, x_9 x_{17}, x_{23})$.
\end{remark}

Now  assume that $\phi ^l_0 \in \Lambda _{Aut ^+}(\g)$.  Recall that $M^n$ admits  a differential form $\phi^l$ of type
$\phi^l_0$, which by Theorem \ref{main3} is multi-symplectic. Let $\xi$ be the intrinsic torsion of the $Aut^+(\g)$-structure on $M^n$. As in the previous sections we denote by $\om ^3$ the Cartan 3-form on $M^n$.

\begin{lemma} \label{exi2} 1. The values of $d\phi^l$ and $d^* \phi^l$ depend linearly on the intrinsic torsion $\xi$ of $M^n$.
\begin{equation}
d{\phi^l} (x) =  \sum _i  e_i \wedge \rho_*(\xi(e_i))( \phi^l(x)), \label{dlx} 
\end{equation}
\begin{equation}
d^*{\phi^l} (x)  =  \sum _i  e_i \rfloor \rho_*(\xi(e_i))(\phi^l(x)), \label{delx} 
\end{equation}
where $(e_i)$  is an orthonormal basis  in $T_xM^n$, see   also (\ref{do1}), (\ref{do2}).\\
2. If $d\phi^l = \phi^l \wedge \theta$ then $\theta$ is defined uniquely by the following formula
\begin{equation}
\theta =  c(\phi^l_0) * ( \phi \wedge * d\phi),\label{eas1}
\end{equation}
for some nonzero constant $c(\phi^l_0)$.\\
3. If $\om^3$ is quasi-closed, i.e. $d\om^3 = \om^3 \wedge \theta$, then  $\om^3$ is locally conformally closed: $d\theta = 0$.
\end{lemma}
\begin{proof}
The first assertion of Lemma \ref{exi2} is a direct consequence of (\ref{rel1}), (\ref{kn1}) and (\ref{kn2}).
In particular, the existence of  a  closed form   or a quasi closed form of type $\phi^l_0$ on
$M^n$ is defined entirely by the algebraic type of the intrinsic torsion  $\xi$ of $M^n$.  

Next we observe that the linear map $L_{\phi^l}: \g \to \Lambda^{l+1} (\g), \: \theta \mapsto \phi^l\wedge \theta$
is an $Ad(\g)$-equivariant map between two $Ad(\g)$-irreducible  modules, moreover $L_{\phi^l}$ extends linearly to
the  complexified irreducible modules of $\g$, since $\g^\C$ is simple. Applying the Schur Lemma we obtain the second
assertion of Lemma \ref{exi2}.

Note that $d^2 \om^3 = 0 = \om^3\wedge \theta\wedge \theta - \om^3 \wedge d\theta =- \om^3 \wedge d\theta$. By (\ref{2om}) proved below  $d\theta = 0$. This completes the proof of Lemma \ref{exi2}.
\end{proof}


\medskip

In what follows we single  out  several  interesting $\g$-modules   associated
with an $Ad (\g)$-invariant  form $\phi^l_0$ in the  exterior algebra $\Lambda (\g)$.
If $\phi^l_0$ is also  $Aut^+ (\g)$- invariant, then  these modules  generate  associated   $Aut ^+ (\g)$-invariant sub-bundles in $M^n$.

Let $\phi^l$ be an $Ad(\g)$-invariant $l$-form on $\g$ of degree  $3\le l\le n-3$ (for  simplicity we drop a lower index $0$ at $\phi^l_0$   when we   are dealing exclusively with $l$-forms on $\g$). For each  $0\le k \le n$ we define the
following linear operator
$$L_{\phi^l} : \Lambda^k (\g) \to \Lambda ^{k +l} (\g) , \:  \gamma \mapsto  \phi ^l \wedge \gamma.$$

Now we look  at a decomposition of $\Lambda^k (\g)$  under the action of $L_{\phi^l}$  for $0\le k \le n$.  Set

\begin{equation}
\Lambda ^k_{\phi^l_+ }: = \{ \beta \in \Lambda ^k (\g)|\,L_{\phi^l} (\beta) = 0 \},
\label{phip}
\end{equation}
\begin{equation}
\Lambda ^k _{\phi^l_-}: = \{ \gamma \in \Lambda ^k (\g)|\,\la \gamma, \beta \ra = 0\,  \forall  \beta \in \Lambda ^k_+ \},
\label{phim}
\end{equation}
 
where  $\la, \ra$  denotes the induced   inner product on $\Lambda ^k (\g)$.  Denote by $*$ the Hodge
operator  on $\g$  associated with the Killing metric  and some preferred orientation on $\g$.  

\begin{proposition}\label{dec1} Assume that $0\le k  \le n$. Then
\begin{equation}
\Lambda ^k_{\phi^l_-} = *L_{\phi^l} (\Lambda ^{n-l-k}(\g)) = * L_{\phi^l} ( \Lambda ^{n-l-k} _{\phi^l_-}).
\label{dual1}
\end{equation}
The space $\Lambda ^k_{\phi^l_-}$ is not zero  if  and only if $(n-l)\ge k$.  The  operator $*L_{\phi^l}$  induces
an isomorphism between  $\Lambda^k_{\phi ^l_-} $  and $\Lambda ^{n-l-k}_{\phi ^l_-}$.  In particular we have
\begin{eqnarray}
\Lambda ^0_{\phi^l_-}  \cong \Lambda^l_{\phi^l_-} = \R.\label{p0l}\\ 
\Lambda ^1_{\phi^l_-}\cong \Lambda  ^{l-1}_{\phi^{n-l}_-} \cong \Lambda ^1(\g).\label{p1l}\\
\Lambda ^{k} _{\phi^l_-} =  0 \text{ for }  n-l +1\le k \le n.\label{pkl}
\end{eqnarray}
Furthermore  we have the following identities
\begin{equation}
(-1)^{l+1}L_{\phi^l} d_\g + d_\g L_{\phi^l} = 0 = (-1)^{l+1}L_{\phi^l} \delta_\g + \delta_\g L_{\phi^l}.\label{sup0}
\end{equation}
Hence operators $L_{\phi^l}$  preserve the subspaces $\Lambda _{harm}(\g), \Lambda _{d}(\g), \Lambda_\delta(\g)$.
\end{proposition}

\begin{proof} First we show that $*L_{\phi^l} (\Lambda ^{n-l-k}(\g)) \subset \Lambda ^k_{\phi,-}$.
Let $\beta \in \Lambda^k_{\phi^l_+}$.
Then
$$ \la \beta, *(\phi^l \wedge \gamma) \ra = \la  vol,  \beta \wedge \phi^l\wedge  \gamma \ra = 0$$
since $\beta \wedge  \phi^l = 0$. Hence  $* (\phi^l \wedge \gamma)\in \Lambda ^k _{\phi_-}$.  

Now we show that $\Lambda ^k _{\phi^l_-} \subset *L_{\phi^l} (\Lambda ^{n-l-k}(\g)) $. 
It suffices to show that  the orthogonal complement of $*L_{\phi^l} (\Lambda ^{n-l-k}(\g))$
in $\Lambda^k(\g)$  is a subset of $\Lambda ^k _{\phi^l_+}$. Clearly 
$$ \beta \in (*L_{\phi^l} (\Lambda ^{n-l-k}(\g)))^\perp \LLR \la \beta , *( \phi^l \wedge \gamma) \ra = 0 \text{ for  all } \gamma\in \Lambda ^{n-l-k} (\g). $$
Then 
$$ \beta \wedge\phi^l\wedge \gamma = 0 \text{ for  all } \gamma \in \Lambda ^{n-l-k} (\g) .$$
Hence
$$ \beta \wedge  \phi^l = 0,$$ 
which by definition  (\ref{phip}) implies that $\beta  \in \Lambda ^k _{\phi^l_+}$.
This proves the first  assertion  (\ref{dual1}) of Proposition \ref{dec1}. 

The second  assertion follows from the definition (\ref{phim}) of 
$\Lambda ^k_{\phi_-}$, taking into account the  existence of  a  nonzero element $\gamma \in\Lambda ^{n-l-k}(\g)$ such that $\phi ^l \wedge \gamma \not = 0$, if $n-l-k \ge 0$. 

The third  assertion follows from (\ref{dual1}), (\ref{phip}), (\ref{phim}).
Clearly the next assertions (\ref{p0l}) and  (\ref{pkl}) are direct consequences of the second and the third assertion. It remains to prove
(\ref{p1l}).  Note that $\Lambda ^1_{\phi^l_-}$ is nonempty by the second assertion.   Since $\Lambda^1 (\g)$ is an irreducible   $\g$-module,  we  get (\ref{p1l}) from the third assertion of Proposition \ref{dec1}.

Next, using $d_\g \phi^l = 0$ we get the first identity in (\ref{sup0}).
To prove the second identity  in (\ref{sup0}) we use the following

\begin{lemma}\label{snl1} \cite{Koszul1985} The  Schouten-Nijenhuis bracket  $\{, \}_\g$ on $\Lambda (\g)$ can be expressed in terms of $\delta_\g$ as follows
$$ \delta_\g (A_l \wedge B_m) =  \delta_\g (A_l) \wedge B_m  +(-1) ^l A_l \wedge \delta_\g(B_m) + (-1) ^{l+1}\{A_l , B_m\}_\g .$$
\end{lemma}
 Substituting $A_l = \phi^l$ in the above formula, and taking into account  $\{\phi^l, B_m\}_\g = 0$ for all $B_m\in \Lambda (\g)$, since $\phi^l$ is $Ad(\g)$-invariant,   we get
\begin{equation}
\delta_\g (\phi^l \wedge B_m) = (-1)^l \phi^l\wedge \delta _\g (B_m), \label{snl2}
\end{equation}
which is equivalent to the second identity in (\ref{sup0}).
This completes the proof of Proposition \ref{dec1}.
 \end{proof}

 
\begin{corollary}\label{omg} The following relations hold for any compact simple Lie algebra $\g$.
\begin{eqnarray}
\Lambda ^2 _{(*\om_\g)_-} = d_\g (\Lambda ^1 (\g)) \cong \Lambda ^1_{(*\om_\g)_-} = \g \text{ as $\g$-module}. \label{striv}\\
\Lambda ^2 _{(\om _\g) _-}\cong \Lambda ^{n-5} _{(\om_\g)_-} \cong \Lambda ^2 (\g) \text{ as $\g$-module}.\label{2om}\\
\Lambda ^3_{(*\om_\g)_-} \cong \Lambda ^0 _{(*\om_\g)_-} \cong \R \text { as $\g$-module}.\label{3oms}\\
\Lambda ^3_{(\om_\g) _-}  \supset d_\g(\g^\perp).\label{3om}\\
\Lambda ^k_{(\om_\g)_+} \supset (\Lambda^{k-3}(\g)\wedge \om_\g).\label{kom} 
\end{eqnarray}
\end{corollary}


\begin{proof} Clearly (\ref{striv}) follows from (\ref{p1l}).

Let us prove (\ref{2om}). By Remark  \ref{ten} $\Lambda ^2 (\g)$ is a sum of two irreducible $\g$-sub-modules $d_\g (\Lambda ^1 (\g))$ and $\g^{\perp}$. Thus it suffices to show that the action of $L_{\om_\g}$ restricted to each 
 sub-module $d_\g (\Lambda ^1 (\g))$ and $\g^{\perp}$  is not zero.

 First we will show that the image $L_{\om_\g} (d_\g(\Lambda ^1( \g))) \not = 0$. 
 Equivalently it suffices to show that $d_\g (L_{\om_\g} \Lambda ^1 (\g)) \not = 0$. By (\ref{sup0}) we have
 $\delta_\g L_{\om_\g} \Lambda^1(\g) = 0$.  Since $b_4 (\g) = 0$ it follows that $L_{\om_\g} \Lambda ^1 (\g))\not\subset \ker d_\g$. Hence  $L_{\om_\g} (d_\g(\Lambda ^1( \g))) \not = 0$.

Next we will show that $L_{\om_\g} (\g ^\perp) \not = 0$. 
Let  $\triangle$ be the root system of $\g ^\C$.   
We  recall  the following root  decomposition  of the complexification
$\g^\C$, where  $\g$ is a compact real form  of $\g^\C$, see e.g. \cite[Theorem 4.2]{Helgason1978} and
\cite[Theorem 6.3]{Helgason1978}. Let $\h_0\in \g$  
be a Cartan subalgebra of $\g$, so $\h^\C_0$ is a Cartan subalgebra  of $\g ^\C$. Let $\triangle ^+$ be a positive root system of $\g ^\C$  and $\Sigma \subset \triangle ^+$ be a system of simple roots. Denote by $E _{\pm\alpha}$, $\alpha \in \triangle ^+$, the   corresponding  root  vectors such that $[E_\alpha ,E_{-\alpha}] =
{2 H_\alpha \over \alpha ( H_\alpha)}\in \h_0 ^\C$, see e.g. \cite[p.258]{Helgason1978}.
We  decompose $\g$ as
\begin{equation}
\g^\C = \oplus _{\alpha \in \Sigma} \la H_\alpha \ra _\R\oplus _{\alpha\in \triangle^+} \la E_\alpha\ra _\R \oplus _{\alpha\in \triangle^+} \la E _{ - \alpha}\ra _\R.
\label{3.4}
\end{equation}
Then
\begin{equation}
 \g = \oplus _{\alpha \in \Sigma}  \la h_\alpha\ra _\R \oplus_{\alpha\in \triangle^+} \la e_\alpha \ra  _\R \oplus_{\alpha\in \triangle ^+}  \la f_\alpha \ra _\R.
\label{3.5}
 \end{equation}

Now we set $h_\alpha : = iH_\alpha$, $e_\alpha := i( E_\alpha + E _{ - \alpha})$  and $f_\alpha :=  ( E _\alpha - E_{-\alpha} )$.

Note that
$ L_{\om_\g} (\delta _\g (h_\alpha \wedge h_\beta \wedge e_\alpha))$ contains a  nonzero  summand of form 
$e_\alpha\wedge e_\beta \wedge e_{\alpha +\beta}\wedge (- c_{\alpha}   h _\beta  + c^\alpha _\beta h_\alpha )\wedge  f_\alpha$, where $c_\alpha c^\alpha _\beta \not = 0$.  
Since $\delta _\g (\Lambda^3(\g))=\g ^\perp$ is  irreducible,  it follows that $\g ^\perp \subset  \Lambda ^2 _{(\om _\g)_-}$. This completes the proof of (\ref{2om}).

Clearly (\ref{3oms}) is a  direct consequence of Proposition \ref{dec1}.

Now let us prove (\ref{3om}). Note that 
\begin{equation}
\la \om _\g \ra _\R  \subset \Lambda ^3 _{(\om _\g)_+}\label{harm3}.
\end{equation}
  Next we show  that 
\begin{equation}
L_{\om_\g} d_\g (\g^\perp) \not = 0. 
\label{l31}
\end{equation}
By (\ref{sup0}) we have  
$$\delta_\g L_{\om_\g} (\g^\perp) = -L_{\om_\g} \delta_\g (\g^\perp) = 0.$$
It follows that $L_\om (\g^{\perp}) \subset \ker \delta_\g$.  Since $\g^{\perp}$ is an irreducible module,
$L_\om (\g^{\perp})$ does not contain any $Ad(\g)$-invariant  form. Hence  $L_\om (\g^{\perp})\subset \Lambda ^5_\delta (\g)$, in particular $d_\g (L_{\om_\g} (\g^\perp)) \not = 0$.  This implies (\ref{l31}).

The last formula (\ref{kom}) is a consequence of the identity $L_{\om_\g} ^ 2 = 0$.
This completes the proof of Corollary \ref{omg}.

\end{proof}


\begin{example}\label{su3} As a consequence of   Proposition \ref{dec1}  and Corollary \ref{omg} we write here the complexes $\Lambda _{(\om_\g)_\pm}$  and $\Lambda _{(*\om_\g)_\pm}$ for $\g = \su (3)$.
\medskip

\begin{tabular}{l|l|l|l|l |l|l|l|l|l}
 a & $\Lambda ^ 0 _a $ & $\Lambda ^1 _a $ & $\Lambda ^2 _a $ & $\Lambda ^3 _a $ & $\Lambda ^4 _a $& $\Lambda ^5 _a $& $\Lambda ^6 _a $& $\Lambda ^7 _a $& $\Lambda ^8 _a $\\
$*\om_\g-$ & $\R$ & $\g$ & $d_\g(\g)$ &  $\la \om_\g\ra $ & 0 & 0 & 0 & 0 & 0\\
$*\om_\g+$ &  0  & 0 & $\g^\perp$ & $\la \om_\g\ra ^{\perp}$ & $\Lambda ^4(\g)$ & $\Lambda ^3 (\g)$ & $\Lambda ^2  (\g)$ & $\g$ & $ \R$ \\
$\om_\g-$ & $\R $ & $\g$ & $\Lambda^2(\g)$ & $d_\g(\g^\perp) \oplus \delta_\g (\Lambda ^4_{(\om_\g) _-})$ & $*(\Lambda^1(\g) \wedge \om_\g)$ & $\la *\om_\g\ra$ & 0 & 0 & 0\\
$\om_\g+$ & 0 & 0 & 0 &  $\la \om_\g \ra_\R\oplus \R^{27}$ & $\R^8 + 2 \R^{27}$ & $\la *\om_\g\ra ^{\perp}$ & $\Lambda^2(\g)$ & $\g$ & $\R$
 \end{tabular}

\medskip

In this example except the modules $\Lambda ^3_{(\om_\g)_\pm}$ all other  modules $\Lambda^i_a$ can be
defined  easily using Proposition \ref{dec1}  and Corollary \ref{omg}. Since $\Lambda ^4(\g)$ is invariant under
the Hodge star operator $*$ we get the following decomposition
\begin{equation}
\Lambda ^4(\g) = (\Lambda ^1(\g) \wedge \om_\g) \oplus \delta_\g (\Lambda ^5_{27}) + *(\Lambda ^1(\g) \wedge \om_\g) \oplus d_\g (\Lambda ^3_{27}),\label{d4}
\end{equation}
where $\Lambda^3_{27}$ is an irreducible $\g$-submodule of $\Lambda ^3(\g) = \la \om_\g\ra _\R \oplus d_\g (\g^\perp) \oplus \Lambda ^3_\delta(\g)$, (so $\Lambda^3_{27} = \Lambda^3_\delta(\g)$) and $\Lambda ^5_{27} = * \Lambda ^3_{27}$. Using (\ref{snl2}) we get $\Lambda^1(\g)\wedge \om_\g \subset \ker \delta_\g$, and  hence $*(\Lambda ^1(\g)\wedge \om_\g) \subset \ker d_\g$.  

Now let us show that
\begin{equation}
\delta_\g (\Lambda^4_{(\om_\g)_-})\subset \Lambda ^3_{(\om_\g) _-}.\label{ex13}
\end{equation}
Using (\ref{snl2}) it suffices to show that $\delta_\g(L _{\om_\g} (\Lambda^4_{(\om_\g)_-}))\not = 0$.
But that is obvious, since $\ker (\delta_\g)_{|\Lambda ^7 (\g)} = 0$. This yields (\ref{ex13}).

In the same way we  get $L_\om (\delta _\g (\Lambda ^4_{(\om_\g)_+} \cap \ker d_\g)) = 0$, which yields 
\begin{equation}
\delta_\g (\Lambda ^4_{(\om_\g)_+} \cap \ker d_\g) = \Lambda ^3 _{27} \subset \Lambda ^3_{(\om_\g)_+}.\label{ex14}
\end{equation}
The modules $\Lambda ^3_{(\om_\g)_\pm}$ can be completely defined from Proposition \ref{dec1}, Corollary \ref{omg}
and (\ref{ex13}), (\ref{ex14}).
\end{example}
\medskip

Now we assume that $\phi^l_0$ is an $Aut ^+ (\g)$-invariant $l$-form. Using   (\ref{phip}) and (\ref{phim}) we define the corresponding  decomposition of 
$$\Om^k(M) = \Om ^k_{\phi^l _+}(M)\oplus \Om ^k_{\phi^l _-}(M).$$
This leads to  the following exact sequence  of modules
\begin{equation}
0 \to \Om ^k_{\phi^l_+}(M)\stackrel{i}{\to} \Om ^k (M)\stackrel{p} {\to}\Om ^k _{\phi^l_-}(M) \to 0 .
\label{spec1}
\end{equation}
Denote by $d_-$ the composition of the differential operator $d$ with the projection $\Pi_-$ to   $\Om_{\phi^l_-}(M)$. 

\begin{theorem}\label{closed}  Assume that  $\phi ^l$ is {\it quasi-closed}. Then \\
 1.   $(\Om ^*_{\phi^l _+}(M), d)$
is a  differential sub-complex of $(\Om^*(M), d)$.\\
2.  $(\Om^*_{\phi^l _+}(M), d_-)$ is a differential complex. 
\end{theorem}

\begin{proof} 1. First   we assume that $\phi^l$ is quasi-closed. Let us show that $(\Om ^*_{\phi^l _+}(M), d)$
is a  differential complex. 
Assume that $\beta \in \Om ^k _{\phi^l_+} (M)$. Then
\begin{equation}
d\beta \wedge \phi^l =  d( \beta \wedge \phi^l) + (-1) ^{deg\, \beta} \beta \wedge d\phi^l = (-1)^{deg\, \beta} \beta \wedge \theta \wedge  \phi^l = 0,\label{cl1}
\end{equation}
which implies that 
\begin{equation}
d\beta \in \Om ^k_{\phi^l_+}(M). \label{cl2}
\end{equation}
 This proves the first assertion of  Theorem \ref{closed}.

2. Next let us show that $d_- ^2 = 0$. Assume that $\alpha \in  \Om ^k _{\phi^l_-}(M)$. 
 Since $ d\alpha = d_- \alpha + \beta$,  where $\beta \in \Om^{k+1}_{\phi^l_+}(M)$,
we get  $d d_-\alpha + d \beta = 0$.  By (\ref{cl2}) $d\beta \in  \Om^{k+2} _{\phi^l_+}(M)$, hence $d d_-\alpha \in \Om^{k+2} _{\phi^l_+}(M)$,  which implies
that $d_ - ^2 \alpha = 0$.  
This completes the proof of Theorem  \ref{closed}.
\end{proof}

Denote  by $H^i(M)$ the de Rham cohomology group $H^i(M, \R)$, and by $H^* _{\phi^l_\pm}(M^n)$  the  cohomology  groups of  modules $(\Om _{\phi^l_\pm}(M), d_{\pm})$, where
$d_+$ is the restriction of $d$ to $\Om_{\phi^l_\pm}(M)$. 

\begin{proposition}\label{spec2}  There exists a long exact sequence of  cohomology groups 
$$0 \to H^1 (M^n) \stackrel{p^*}{\to}  H^1 _{\phi^l_-} (M^n)  \stackrel{t^*}{\to} H ^2 _{\phi^l_+} (M^n) \stackrel{i}{\to} H^2 (M^n) \stackrel{p^*}{\to }\cdots H^n (M^n) \to 0.$$
\end{proposition}

\begin{proof}  Let us define the following diagram of chain complexes
$$
\xymatrix{\Om ^k_{\phi^l_+}(M) \ar[r] ^ {d_+} \ar[d] ^i &  \Om^{k+1}_{\phi^l_+}(M) \ar[r] ^ {d_+} \ar[d] ^i & \Om^{k+2}_{\phi^l_+}(M) \ar[r] ^ {d_+} \ar[d] ^i &\\ 
\Om ^k(M) \ar[r] ^ {d} \ar[d] ^p &  \Om^{k+1}(M) \ar[r] ^ {d} \ar[d] ^p & \Om^{k+2}(M) \ar[r] ^ {d} \ar[d] ^p &\\ 
\Om ^k_{\phi^l _-}(M) \ar[r] ^ {d_-}  &  \Om^{k+1}_{\phi^l_-}(M) \ar[r] ^ {d_-}  & \Om^{k+2}_{\phi^l_-}(M) \ar[r] ^ {d_-} &}
$$
 To prove Proposition \ref{spec2}, taking into account (\ref{spec1}), it suffices to show that  the  above diagram  is commutative.  Equivalently we need to show 
\begin{equation}
d \circ i = i \circ  d_+, \,  d_- \circ p  = p \circ  d. \label {com1}
\end{equation}
The first identity in (\ref{com1}) is a consequence of (\ref{cl2}) (the  $d$-closedness of $\Om _{\phi ^l _+}(M)$).
The second identity follows from the definition of $d_-$. This completes the proof of Proposition \ref{spec2}.
\end{proof}

\begin{remark} The map $p^* : H^k(M) \to H^k _{\phi^l_-} (M)$ associates  each  cohomology class of a closed $k$-form  $\phi^k$ on
$M$ to  the cohomology class  $[p (\phi ^k)] \in H^k_{\phi^l_-} (M)$, see (\ref{com1}). The connecting homomorphism
$t^* : H ^k_{\phi^l_-}(M)\to H^{k+1} _{\phi ^l_+} (M)$ is defined as follows:
$[\phi] \mapsto  [d\phi]$. The  map $i ^* : H^k_{\phi^l_+}(M) \to   H^k(M)$ associates each cohomology class in $H^k_{\phi^l _+}(M)$ of
a closed $k$-form $\phi^l \in \Om ^k _{\phi^l} (M)$ to the cohomology class $[\phi^k]\in H^k(M)$.
\end{remark}

\begin{example}\label{csu3} Let us consider the  new homology groups arisen  in this way on a connected orientable manifold $M^8$  provided with a 3-form $\om^3$ of  type $\om_{\su(3)}$ assuming that $d\om^3 = 0 = d^* \om^3$.
\medskip

\begin{tabular}{l|l|l|l|l}
a & $*\om_\g-$ & $*\om_\g+$  & $\om_\g -$  & $\om_\g +$\\
$H ^ 0 _a $  & $\R$ &  0 &  $\R$ & 0 \\
$H ^1 _a $ & $H^1(M) \oplus \frac{\{d^{-1}(\Om^2_{(\om_\g)_+})\}}{\{\ker d\cap \Om^1(M)\}}$&  0& $H^1(M)$ &0\\
$H ^2 _a $ &  $\frac{\Om^2_{(*\om_\g)_-}}{\Pi_{d_\g} (d\Om^1(M))}$&  $\{\ker d \cap \Om^2_{(*\om_\g)_-}\}$&  $H^2(M) \oplus \frac{\{\phi|d\phi\wedge \om_\g = 0\}}{\{\phi|d\phi =0\}}$ & $0$\\
$H^3 _a $  & $\R $ & $\frac{H^3(M)}{\R} \oplus d\Om^2_{(*\om_\g) _-}$ &   $H^3 _{(\om_\g)_- }$  & $\{\ker d\cap\Om^3_{(\om_\g)_+}\}$\\
$H ^4 _a$ & 0 & $H^4(M)$ & $\frac{\Om^4_{(\om_\g) _-}}{d_-(\Om^3_{(\om_\g)_-})}$ & $H^4_{(\om_\g) _+}(M)$\\
$H ^5 _a $&  0&  $H^5 (M)$ & $\R$ &  $\frac{H^5 (M)}{\R} \oplus d\Om^4_{(\om_\g)_-}$\\
$H ^6 _a $&  0& $H^6(M)$ &  0 & $H^6(M)$ \\
$H ^7 _a $&  0& $H^7(M)$ &  0 & $H^7(M)$ \\
$H ^8 _a $  & 0 & $\R$ & 0  & $\R$\\
\end{tabular}
\medskip

In this table we denote  by $\Pi_{d_\g}$ the orthogonal projection of $\Om^*(M)$ on the image of  the operator $d_\g$ acting on $\Om^*(M)$.
The cohomology groups $H^3_{(\om_\g)_-}$, $H^4_{(\om_\g)_+}$  are best described by using the relations in Example
\ref{su3}.  
\end{example}

\begin{remark}\label{fin1} Note that many  cohomology groups  $H^i_{\phi^l_\om}$ are infinite-dimensional.
 Proposition \ref{spec2}  implies that  the subgroup $p^*(H^i(M)) \subset H^i_{\phi_-}(M^n)$ as well as the coset
 $H^{i+1}_{\phi^l_+}/t^*(H^i_{\phi^l_-}(M))$ are finite-dimensional.
 \end{remark}

Now we are going to define a subgroup of $H^i_{\phi^l_\pm}(M^n)$.




\begin{definition}\label{har} A differential form $\alpha$ in $\Om^* _{\phi_\pm}(M)$  is called {\it $\phi_\pm$-harmonic}, 
if $d_\pm \alpha = 0 = d^* \alpha$. 
\end{definition}

Denote by  $\Hh^k (M^n)$  the space of all  harmonic $k$-forms on $M^n$. Let $\Hh_{\phi^l_\pm} ^ i (M)$ be the space of all  $\phi^l_\pm$-harmonic  forms  in $\Om ^k_{\phi^l_\pm}(M^n)$. 

\begin{proposition}\label{har2}  There is a natural monomorphism  $h_i^\pm: \Hh _{\phi^l_\pm} ^ i (M) \to H_{\phi^l_\pm}^i(M)$
associating  each  $\phi^l_\pm$-harmonic form $\alpha$ to  its cohomology  class $[\alpha]\in H_{\phi^l_\pm}^i(M)$. 
\end{proposition}

\begin{proof} Assume that  $\alpha \in \ker h_i^\pm$, i.e. $\alpha$  is a  $\phi_\pm$-harmonic form, and $\alpha = d_\pm \gamma$,
for some $\gamma \in \Om ^{i-1}_{\phi^l_\pm}(M^n)$. By the definition of $d_\pm$  we have
$$ \la \alpha,  d_\pm\gamma\ra = \la\alpha, d\gamma\ra =  \la d^* \alpha, \gamma\ra = 0 $$
which implies that $\alpha = 0$. This proves  Proposition \ref{har2}.
\end{proof}

By Proposition \ref{dec1},  $*L_{\phi^l}$ induces a  bundle  isomorphism $\Om ^1_{\phi^l_-}(M^n)  \to \Om ^{n-l-1}_{\phi^l _-}(M^n)$.
  
\begin{proposition}\label{l1} Assume that $d\phi^l =0$. The operator $*L_{\phi^l}$ induces an isomorphism also denoted by $*L_{\phi^l}:\Hh ^1_{\phi_-}(M^n)  \to  \Hh ^{n-l-1} _{\phi_-} (M^n)$.
\end{proposition}

\begin{proof} First  we show that if $\alpha \in  \Hh ^1_{\phi_-}(M^n)$, then $*L_{\phi^l}(\alpha)\in\Hh ^{n-l-1} _{\phi_-} (M^n)$. We compute
\begin{equation} 
d^* (*L_{\phi^l}(\alpha)) = d (\alpha \wedge \phi^l) =  (d_-\alpha + d_+ \alpha) \wedge \phi^l  = 0, \label{l2}
\end{equation}
since $d_-\alpha = 0$, and $(d_+\alpha) \wedge  \phi^l = 0$. 
 
Next we prove that $d_- (*L_{\phi^l}(\alpha)) = 0$, or equivalently $d(*L_{\phi^l}(\alpha)) \in  \Om ^{n-l}_{\phi^l_+}(M^n)$.  Note  that  we have  the following  orthogonal  Hodge decomposition

\begin{equation}
\Om ^{n-l}(M^n) = \Hh ^{n-l} (M^n) \oplus d (\Om ^{n-l-1} (M^n))   \oplus d^* (\Om ^{n-l+1} (M^n)).\label{hod2}
\end{equation}

Since $d\phi ^l = 0$,   we have $\phi^{n-l} \in \Hh^{n-l}(M) \oplus d^* (\Om ^{n-l+1} (M^n))$.  Hence we get  from  \ref{hod2}
\begin{equation}
\la  \phi ^{n-l}, d(*L_{\phi^l}(\alpha))  \ra = d(*L_{\phi^l}(\alpha))\wedge \phi ^{l} = 0.
\label{hod3}
\end{equation}

It follows that $d(*L_{\phi^l}(\alpha)) \in  \Om^{n-l}_{\phi^l_+}(M^n)$.   Thus $*L_{\phi^l}(\alpha)\in\Hh ^{n-l-1} _{\phi_-} (M^n)$.

By Proposition \ref{dec1} the restriction of $*L_{\phi^l}$ to $\Hh^1_{\phi^l_-}(M^n)$  is 
injective.

Let us show that  this map is also surjective.  We will  first prove  the following

\begin{lemma}\label{irr6}  There exists a nonzero constant $c$  depending on $\phi^l_0$ such that  for all
$\beta \in \Om ^{n-l-1} _{\phi^l_-}(M^n)$  we have
\begin{equation}
*(\beta \wedge \phi^l_0)\wedge \phi^l_0  = c\cdot ( * \beta). 
\end{equation}
\end{lemma}
\begin{proof}  Lemma \ref{irr6} follows from  the fact that the map
$$\Xi: \Lambda^{n-l-1}_{(\phi^l_0)_-} \to \Lambda^{n-l-1}(\g), \: \beta \mapsto * ( * (\beta  \wedge \phi ^ l _0 ) \wedge \phi ^l_0),$$
is an $Ad(\g)$-equivariant map: $Ad(g) \Xi = \Xi Ad(g) $ for all $g \in Ad(\g)$.
By Proposition \ref{dec1} the image $*L_{\phi^l_0}(\Lambda ^{n-l-1}_{(\phi^l_0)_-}) = \Lambda ^1 (\g)$, and apply
 Proposition \ref{dec1} again,  we conclude that the image of $\Xi$ is $\Lambda^{n-l-1}_{(\phi^l_0)_-}$.
Thus $\Xi$ is an $Ad(\g)$-equivariant  endomorphism of  $\Lambda^{n-l-1}_{(\phi^l_0)_-}$.
The same argument as in the proof of Lemma \ref{exi2} using Schur's Lemma implies that $\Xi$ is a  multiple of the identity map. This proves Lemma \ref{irr6}. 
\end{proof}

 Let us continue the proof of Proposition \ref{l1}.  Suppose  that $\beta \in \Hh_{\phi^l_-} ^{n-l-1}$.  Set
\begin{equation}
\alpha := *(\beta \wedge \phi^l) \in \Om ^1 (M^n).\label{inv1}
\end{equation}
Using $d* \beta = 0$, $d\phi ^l = 0$ and taking into account Lemma \ref{irr6} we get
$$ d\alpha \wedge \phi^l = d(\alpha \wedge \phi^l)=  d ( * (\beta \wedge \phi) \wedge \phi) = c\cdot d* \beta = 0,$$
which implies that $d\alpha \in \Om ^2_{\phi^l_+}(M^n)$. Hence
\begin{equation}
 d_-\alpha = 0.\label{l11}
 \end{equation}
Next we note that
\begin{equation}
d*\alpha = d(\beta \wedge \phi^l) = d\beta \wedge \phi = 0\label{l12}
\end{equation}
since $d_-\beta = 0$.

By (\ref{l11}) and (\ref{l12})  $\alpha$ is $\phi^l_-$-harmonic.  This completes the proof
of Proposition \ref{l1}.
\end{proof}

\begin{lemma}\label{har1} Let $M^n$ be a connected compact $Aut(\g)$-manifold  provided  with a  $l$-form  $\phi ^l$ of type $\phi^l_0$. Then\\
1. There  is a monomorphism $ i: H^1 (M^n) \to \Hh_{\phi^l_-} ^1 (M^n)$. This monomorphism is an isomorphism if
$\phi^l_0 = \om^3_\g$.\\
2. If $\delta \phi^l=0$, then $\Hh_{\phi^l_-} ^{n-l} (M^n) = 1$.
\end{lemma}

\begin{proof} 
Let $\alpha$ be a harmonic  1-form on $M^n$. Then $\alpha$ is also a $\phi^l_-$-harmonic form. Thus
$ i: H^1 (M^n) \to \Hh_{\phi^l_-} ^1 (M^n)$ is a monomorphism.  If $\phi^l_0 = \om_\g$ then $i$ is an isomorphism,
since $H^1 (M^n) = H^1_{(\om_\g)_-}$ as a direct consequence of Corollary \ref{omg}, taking into account Proposition
\ref{har2}. This proves the first assertion
of Lemma \ref{har1}.

By Proposition \ref{dec1} we can write  $\Om ^{n-l} _{\phi^l_-}(M^n) = \la* \phi^l\ra _\R$, hence
\begin{equation}
\Hh^{n-l}_{\phi^{l}_-} (M^n) = \{ f\phi \in \Lambda ^{n-l}_{\phi^l_-}|\,  d (f *\phi^l) = 0 \}.\label{harl}
\end{equation}

Since $\phi^l$ is harmonic, (\ref{harl})  holds if and only if $f$ is constant.  This completes the proof
of Lemma \ref{har1}.
\end{proof}

\begin{remark}\label{sal} The construction of our differential complexes is similar to the construction of  differential complexes on a 7-manifold with an ``integrable" $G_2$-structure  by Fernandez and Ugarte in \cite{FU1998}. Their  consideration is based on  Reyes work \cite{Reyes1998}, where   Reyes  also  considered    differential complexes  associated to certain  $G$-structures, using some ideas in \cite{Salamon1989}. Many special properties
of our complexes are  related to the cohomological structure of a simple compact Lie algebra $\g$.  A subcomplex $\Om^2_{(*\om_\g)_+}$ will be shown to play a roll in the geometry of  strongly associative submanifolds of dimension 3, see
Proposition \ref{triv5}.
\end{remark}

\section{Special submanifolds in $Aut(\g)$-manifolds}

In this section we also assume that $\g$ is a compact simple Lie algebra  and $M^n$ is a smooth manifold
provided with a 3-form of type $\om_\g$, where $\g$ is a compact  simple Lie algebra.  The Lie bracket on $\g$ extends
smoothly to a cross-product $TM\times TM \to TM$, which  we denote by $[, ]_\g$. We study a natural class of submanifolds in  $Aut(\g)$-manifolds, which generalize the  notion of Lie subgroups in a Lie group $G$,  and show
their algebraic  and geometric rigidity  (Propositions \ref{main5}, \ref{triv4},  \ref{triv5},  \ref{triv6}, Remark \ref{liu}).

\begin{definition}\label{ass} Let $M^n$  be a manifold provided with a 3-form $\om^3$ of  type $\om_\g$.
A submanifold $N^k\subset M^n$ is called {\it associative}, if the tangent  space $TN^k$ is closed under the bracket
$[, ]_\g$.  A submanifold $N^k \subset M^n$ is called {\it coassociative}, if its normal bundle $T^\perp N^k$ is closed
under the Lie bracket $[, ]_\g$.  An associative (resp. coassociative) submanifold $N^k \subset M^n$ is called {\it strongly associative} (resp.
{\it strongly coassociative}), if the restriction of $\om^3$ to $TN^k$ (resp. to $T^\perp N^k$) is multi-symplectic.
\end{definition}

The following Proposition shows the algebraic rigidity of strongly associative and strongly coassociative submanifolds.  It is a direct consequence of Lemma \ref{multi}, so we omit its proof.

\begin{proposition}\label{main5} Assume that $N^k$ is  a strongly associative submanifold in $M^n$. Then
all the tangent spaces $(T_xN^k, [,]_\g)$  are isomorphic to a compact semisimple Lie subalgebra $\h \subset \g$.
Assume that $N^k$ is a strongly coassociative submanifold in $M^n$. Then all the normal spaces $(T^\perp _x N^k, [,]_\g)$
are isomorphic to a compact semisimple Lie subalgebra $\h \subset \g$. 
\end{proposition}


We  call the semisimple Lie algebra $\h$ described in Proposition \ref{main5}  {\it the Lie  type} of a strongly
associative (resp. strongly coassociative) submanifold $N^k$.

The following Lemma is well-known, so we will omit its proof.

\begin{lemma}\label{kero} Assume that $N^k$ is an associative submanifold in a manifold $M^n$ provided with a closed 3-form $\om^3$. Then the distribution $\ker \om^3_{| N^k}$ is integrable.
\end{lemma}



\begin{proposition}\label{triv4}
1. Assume that $N^3$ is a strongly  associative 3-submanifold in a manifold $M^n$ provided  with a 3-form $\om^3$ of type $\om_\g$. Then  the mean curvature $H$ of $N^3$ satisfies the following equation for any $x\in N^3$
\begin{equation}
H(x)^* = c \cdot \vec{T_xN^3}\rfloor d\om^3(x),\label{min4}
\end{equation}
where  $H(x)^*$ denotes the  covector in $T^*_xM$ dual to $H(x)\in T_xM$ with respect to the Riemannian metric $K_\g(x)$, and $c$ is a positive constant  depending only on the Lie type of  $N^3$.

2. Assume that $N^{n-3}$  is  strongly coassociative submanifold of codimension 3 in a manifold $M^n$ provided  with a  3-form $\om^3$ of type
$\om_\g$. Then  the mean curvature $H$ of $N^{n-3}$ satisfies the following equation for  any $x\in N^{n-3}$
\begin{equation}
H (x)^*  = c\cdot \vec{T_xN^{n-3}}\rfloor  d(*\om^3)(x),\label{min6}
\end{equation}
where $H(x)^*$ denotes the  covector in $T^*_xM$ dual to $H(x)\in T_xM$ with respect to the Riemannian metric $K_\g(x)$, and $c$ is a positive constant depending only of the Lie type of $N^{n-3}$.
\end{proposition}

\begin{proof}  Let us choose $c^{-1} = \om_\g (\vec{T_xN^3})$  for strongly associative  submanifold $N^3$, $x\in N^3$, and
let us choose $c^{-1} = *(\om_\g) (\vec {T_x N^{n-3}})$ for strongly coassociative submanifold $N^{n-3}$, $x\in N^{n-3}$.  By  Theorem 3.1 in \cite{Le1990} $c$ is a nonzero critical value of the function $f_{\om_\g}(x)$  (resp.
$f_{*\om_\g}(x)$) defined on each  Grassmannian $Gr_3^+ (T_xM^n)$ of oriented 3-vectors, $x\in N^3$,  (resp. on $Gr_{n-3}^+ (T_xM^n)$,
$x\in N^{n-3}$)  by setting $f_{\om_\g} (u): = \om_\g(x) (\vec u)$ (resp. $f_{\om_\g} (u): =* \om_\g(x) (\vec u)$).
 
Now Proposition \ref{triv4} follows directly from   \cite[Lemma 1.1]{Le1990}, where we computed
 the first variation  formula for  a Riemannian submanifold $N^k\subset M$ satisfying  the condition that  there is a differential form $\om$ on $M$  such that $\om_{| N^k} = vol_{|N^k}$ and moreover the value $\om(x)(\vec T_xN^k)$ is a nonzero critical value
 of the function $f_\om(x)$ defined on the   Grassmannian $Gr_k^+ (T_xM))$ by  setting $f_\om (u): = \om (\vec u)$. For a strongly associative submanifold $N^3$ Lemma 1.1 in \cite{Le1990}  yields
 
\begin{lemma}\label{lem1v} (cf \cite[Lemma 1.1]{Le1990}) For any point $x \in N^3$   and for any normal vector $X\in T^\perp _xN^3$ we have
\begin{equation}
\la -H(x), X \ra = c \cdot d\om^3 (X\wedge \vec{T_xN^3}).
\label{min5}
\end{equation}
\end{lemma}

Clearly (\ref{min4}) is  equivalent to (\ref{min5}).  In the same way we prove (\ref{min6}).
\end{proof}

Strongly associative  or coassociative submanifolds  of (co)dimension 3  thus satisfy the 
equation in Proposition \ref{triv4}, which is a first order  perturbation of the
second order elliptic equation describing  minimal submanifolds. 
On the other hand, as in \cite[Proposition 2.3]{Robles2008} we can also define strongly associative   submanifolds (resp. strongly coassociative manifolds) as integral submanifolds of some differential system on $M$. 

\begin{proposition}\label{triv5}  A 3-submanifold  $N^3\subset M^n$ is strongly associative, if   and only if the restriction of any 3-form on $M$ with value in the subbundle $d_\g (\Om^2_{(*\om_\g) _+})$ to $N^3$ vanishes.
A submanifold $N^{n-3}\subset M^n$  is strongly coassociative, if the restriction of  any
$(n-3)$-form on $M$ with value in the subbundle $*d_\g (\Om^2_{(*\om_\g) _+})$ to $N^{n-3}$ vanishes.
\end{proposition}

\begin{proof}  The first  assertion of Proposition \ref{triv5} is a direct consequence of \cite[Proposition 2.3]{Robles2008}, which describes the set of $\phi$-critical planes, applying  to $\phi = \om_\g$   taking into account the formulas  (\ref{dg2}) and (\ref{striv}). 

 The second assertion is a consequence  of  \cite[Proposition 2.3]{Robles2008}, the formula
(\ref{dg2})  and the following  identity
$$\rho_*(\tau)( * \om_\g) = * (\rho_* (\tau) \om_\g)$$
for any $\tau \in \so (\g) \cong \Lambda ^2 (\g)$.
\end{proof}

\begin{remark}\label{simcal} Propositions \ref{triv4} and \ref{triv5} show that  strongly associative  submanifolds in dimension 3 and
strongly coassociative submanifolds in co-dimension 3 like  calibrated submanifolds  are
 solutions  to a  special  class of first order  PDE, which  are in the same time  solutions to a special class of  elliptic second order equations.
 In particular, if  $d\phi^k = 0$, then  strongly (co)associative  submanifolds of  (co)dimension 3  in Propositions \ref{triv4} and \ref{triv5}  are calibrated submanifolds.
\end{remark}

 Proposition \ref{triv4} can be generalized for other strongly (co)associative  submanifolds $N^k$, when $k \ge 4$, under some additional conditions  on the Lie type of $N^k$. We say that a compact Lie subalgebra $\h \subset \g$ is {\it critical  with respect to an $Ad(\g)$-invariant $k$-form $\phi^k_0$ on $\g$}, if $\vec{\h}$ is  a critical point corresponding to a nonzero critical value
 of the function $f_{\phi^k_0}$ defined on $Gr _k ^+(\g)$ as we defined above for $\phi^k_0 = \om_\g$.

\begin{example}  In \cite{Le1990} we showed  that the canonical embedded Lie subalgebra $\su(n) \to  \su(m)$, $n \le m$,
is critical  with respect to  an $Ad(\g)$-invariant  form $\phi^{n^2 -1}$  which is  the wedge product of elementary
$Ad(\g)$-invariant $(2k-1)$-forms $\theta ^{2k-1}$  defined on $\su(m)$ as follows
\begin{eqnarray}
\theta ^{2k-1} (X_1, \cdots, X_{2k-1}) =Re \sum_{\sigma \in \Sigma_{2k-1}} \eps_\sigma Tr (X_{\sigma (1)} \cdots X_{\sigma(2k-1)}) \text{ if  $k$ is even},\nonumber \\
\theta ^{2k-1} (X_1, \cdots, X_{2k-1}) =Im \sum_{\sigma \in \Sigma_{2k-1}} \eps_\sigma Tr (X_{\sigma (1)} \cdots X_{\sigma(2k-1)}) \text{ if  $k$ is odd},\nonumber
\end{eqnarray}
where  $X_i \in \su(m)$  and the multiplication is  the usual  matrix multiplication.

In \cite{Le1990} and \cite{Le1990b}  using different methods we proved that  the canonical embedded  group
$SU(n) \to SU(m)$ is  stable minimal with respect to the bi-invariant Riemannian metric on $SU(m)$.
\end{example}
 
\begin{proposition}\label{triv6} Assume that  a Lie algebra  $\h\subset \g$ is a critical with respect to an $Aut^+(\g)$-invariant  $k$-form $\phi^k_0$   on $\g$.  Then any  associative submanifold $N^k$  of  Lie type $\h$ satisfies
\begin{equation}
H(x)^* =  \vec{T_xN^k}\rfloor d\phi^k(x),\label{min4}
\end{equation}
where  $H(x)^*$ denotes the  covector in $T^*_xM$ dual to $H(x)\in T_xM$ with respect to the Riemannian metric $K_\g(x)$, and $\phi^k$ is the extension of $\phi^k_0$ on $M^n$. Moreover $N^k$ is orientable.
A similar statement also holds for  coassociative  submanifolds $N^{n-k} \subset M^n$.
\end{proposition}

\begin{proof} We  can   assume that 1 is a critical value of
the function $f_{\phi^k_0}$ defined on $Gr _k^+ (\g)$  and $\h$ is a critical point corresponding to this value.    Now applying \cite[Lemma 1.1]{Le1990} to
$N^k$ we obtain Proposition \ref{triv6} immediately. (The orientability of $N^k$ is a consequence of the fact that the
restriction of $\phi^k$ to $N^k$ is  a nonzero form of top degree).
\end{proof}

\begin{remark}\label{liu} As a consequence of a result  by Robles \cite[Lemma 5.1] {Robles2008}, the differential ideal generated by  differential forms  taking value in $d_\g (\Om^2_{(\om_\g)_+})$ (resp. in $ *d_\g (\Om^2_{(\om_\g)_+})$  is differentially closed, if $M^n$ is  torsion-free. Robles also  shows that in the torsion-free case the  differential system  corresponding to coassociative submanifolds of codimension 3 is not in involution.
It is possible that the considered differential system is not minimal  in the sense that there is  a smaller differential system  having the same  integral submanifolds, see \cite[Remark 5.2]{Robles2008}.
\end{remark}


\section{Final remarks}

 1. The richness of geometry of   manifolds equipped with a simple Cartan 3-form  demonstrated in our note  shows that these manifolds could be considered as  natural generalizations of the notion of compact simple Lie groups in the category of Riemannian manifolds. There are many interesting 1-order differential operators on   such manifolds,  which we could exploit further to understand the geometry of the underlying manifolds.  One of  the  operators we have in our mind is the elliptic self-adjoint operator $d + d^* + \lambda ( d_\g + \delta_\g)$, $\lambda \in \R$, acting on $\Om^*(M)$. It would be interesting to develop the theory further for  manifolds of special   algebraic type, especially to find topological constrains of  manifolds with harmonic $\phi^l$-forms.




2.  We have discussed in this note only   first order invariants of considered manifolds  and do not discuss the curvature of the underlying Riemannian metric as well as relations between    first oder invariants and  second order  invariants   of these manifolds.

3. It would be interesting to study  Lie group actions and moments maps  on manifolds with a closed  form  of type
$\phi ^l$, see also \cite{MS2010}.

4. Find sufficient and necessary conditions  to ensure the local existence   of a class of (co)associative submanifolds.



\section{Appendix. Riemannian manifolds equipped with a parallel 3-form}

In this appendix we describe  simply-connected complete Riemannian manifolds provided with  a parallel 3-form. Let us first recall the following basis definitions.

A 3-form $\om ^3 $ on $\R ^ 6 = \C^ 3$ is called {\it a Special Lagrangian 3-form}  (or $SL$-form), if
$\om$ can be written as $\om ^3 =  Re (dz^1 \wedge dz^2\wedge dz^3)$.

A 3-form $\om ^3$ on $\R^7 = Im\, \O$ is called {\it of $G_2$-type},  if $\om^3$ is on the  $GL(\R^7)$-orbit of the form $\om^3_0 (X, Y, Z): = 
\la XY, Z \ra$, where  $\la, \ra$  is the inner product on the octonion algebra $\O$.

A 3-form $\om ^3$ on $\R^{2n+1}$ is called  {\it of product type of maximal rank}, if $\om^3$ can be written
as  $\om ^3 = dz \wedge (\sum _{i =1} ^n dx^i \wedge dy ^i)$.

A 3-form $\om ^3$ on $\R^n$ is called {\it a compact simple Cartan form}, if $\om ^3 = \om_\g$ for
some compact simple Lie algebra $\g$.



\begin{theorem}\label{3para}  Let $M^n$ be a connected simply connected complete Riemannian manifold provided with
a parallel 3-form $\om^3$. Then $(M^n, \om ^3)$ is a direct product of basis Riemannian manifolds $M_i$  provided with   a paralell $k$-form $\om _i$ of the following types.\\
1) A  Calabi-Yau 6-manifold  $M^6$ provided with a $SL$ 3-form.\\
2) A (torsion-free) $G_2$-manifold  $M^7$ provided with a 3-form of $G_2$-type.\\
3) A  compact simple Lie group or its noncompact dual  with the associated Cartan 3-form.\\
4) A  K\"ahler manifold $M^{2n}$ with   a K\"ahler form $\om ^2$.\\
5) A  Euclidean space $(\R^k, \phi^3)$ provided with a parallel multi-symplectic 3-form $\phi^3$.\\
6) A Riemannian manifold  $N^l$  with the zero  3-form. \\
The 3-form $\om^3$ is a  sum of the $SL$-3-forms,  $G_2$-forms, a multiple by a nonzero constant of the Cartan 3-forms, and    3-forms  of
product type of maximal rank, whose precise description will be given in the proof below. 
\end{theorem}

\begin{proof} Let $G$ be the holonomy  group of  a connected simply connected Riemannian manifold $M^n$ provided with
a parallel 3-form $\om^3$.  By the de Rham theorem, see e.g. \cite[10.43, chapter 10]{Besse1987}, $M^n$ is a product of Riemannian manifolds $M_i, i \in I,$  such that
$G= \Pi_{i\in I} G_i$ where $G_i$ acts irreducibly on $TM_i$ for $i\ge 1$ and $M_0$  is flat.  Now  choose an arbitrary 
point $x \in M^n$. Let $L_{\om^3}$ be the map defined   in (\ref{multi}). Since $\ker L_{\om^3}$ is a $G$-module, we have the following decomposition
$$\ker L_{\om^3} = \oplus _{i \in J}T_xM_i \oplus  (\ker L_{\om^3} \cap T_xM_0),$$
 where $J$ is some  subset of $I$.  Let $N$ be  a submanifold in $M$ satisfying the following conditions:\\
1.  $N\ni x$ and $T_x N = \ker L_\om$,\\
2. $N$ is invariant under the action of $G$.

Note that there is a unique  submanifold $N\subset M$ satisfying the conditions 1 and 2.
Let $W$ be the orthogonal complement to $\ker L$.  Denote by $\R^p$ the  subspace of  the flat  space $M_0$ which is tangent to $W$. Set
\begin{equation}
M^n_m: = \Pi_{i \in I \setminus J} M_i \times \R^p. \label{decoma}
\end{equation}
Then $M^n = N \times  M^n_m$.  Let $\pi_m: M^n \to M^n_m$ be the natural projection. Clearly 
$\om^3 \in \pi^*(\Om^3(M^n_m))$.
Denote by $\bar \om_i^3$ the restriction of
$\om^3$ to $M_i$ which enter in the decomposition (\ref{decoma}). Since $G$ preserves $\om^3$, the subgroup $G_i$ preserves the 3-form $\bar\om^3_i$.

\begin{lemma}\label{3irr}  If $\bar \om^3_i$ is not zero  and $G_i$ acts irreducibly on $TM_i$ then $\bar \om_i^3$ is either a SL 3-form, or a $G_2$-form, or a multiple of a Cartan simple
form.
\end{lemma}

\begin{proof}  Note that  $G_i$ is  a  subgroup of the stabilizer of $\bar\om^3_i$.  First we inspect the list of all possible  irreducible Riemannian holonomy groups  in \cite[table 1, chapter 10]{Besse1987}, using the fact that the groups $U(n), SU(n)$ (resp. $Sp(n), Sp(1)Sp(n)$)  have no invariant nonzero 3-form on the space $\R^{2n}$, if $n\not = 3$ (resp. on the space $\R^{4n}$  for all $n$), as well as  $Spin(7)$ has no-invariant 3-form on $\R^8$,
to conclude  that $M_i$ must be a symmetric space of type I or IV, or a Calabi-Yau 6-manifold, or a $G_2$-manifold.  Using the  table of Poincare polynomials for symmetric spaces of type 1  in \cite{Takeuchi1962} we obtain that
the only  symmetric spaces of type 1 and  of dimension greater than equal 3  with non-trivial  Betti number $b_3$  are compact simple Lie groups.  This completes the proof of Lemma \ref{3irr}.
\end{proof}

Now we assume that $\bar \om ^3_i = 0$. 
We say that the 3-form $\om^3$ has {\it rank 2} on $M_i$, if for some $x\in M_i$  (and hence  for all $x \in M_i$)
the linear map $L^2_{\om^3}: \Lambda ^2 T_xM_i \to T^*_xM, v\wedge w \mapsto (v\wedge w) \rfloor \om^3$, has nonzero
image. 
We say that $\om^3$ has {\it rank 1} on $M_i$ if the image $L^2_{\om^3}$ is zero and  for some $x\in M_i$
(and hence for all $x\in M_i$)
the map $L_{\om^3}:  TM_i \to T^*M, v\mapsto v \rfloor \om^3$ has nonzero
image.  If both the maps $L^2_{\om^3}$ and $L_{\om^3}$ have trivial image, then  clearly $\om^3$  belongs to the
space $\Lambda ^3 (T^*(M_1 \times \cdots, _{\hat i} \cdots \times M_k))$.

\begin{lemma}\label{product} Assume that $\bar \om^3_i = 0$. Then the 3-form  $\om^3$ has   rank 2 on $M_i$ and   $M_i$ is a K\"ahler manifold provided with a K\"ahler 2-form $\om^2_i$. Moreover  there is a subspace $\R_i \subset \R^p$ provided with a constant 1-form $dx^i$ such that
the restriction of $\om^3$ to $M_i \times \R_i$ is  equal to $\lambda_{i} dx^i\wedge \om^2_i$  for some nonzero constant $\lambda_{i}$.
\end{lemma}

\begin{proof}  Assume that  $\om^3$ has rank 1 on $M_i$. 
Then there is  a vector $v\in T_xM_i$  and two vectors  $u, w \in T_xM_m^n$ such that $u, v$ are orthogonal to $T_xM_i$ and
$\om^3(v, u, w) = 1$. Since $u, v$ are orthogonal to $T_xM_i$ the 1-form $(u\wedge w) \rfloor \om^3$ is invariant under
the action of $G_i$.  Taking into account the irreducibility of the action of  $G_i$ on $T_xM_i$ we  obtain the first assertion of Lemma \ref{product}.  Thus $\om^3$ has rank 2 on $M_i$, if $\bar \om^3_i = 0$.
It follows that there is a vector $v\in T_xM^n_m$ such that $v$ is orthogonal to $T_xM_i$ and
$\om^3 (v, u, w) =1$ for some two vectors $u,w \in T_xM_i$. Repeating the previous argument, we conclude that
$v\in T_x\R^p$ and 
$G_i$ leaves  $\Om_i: =(u\rfloor\om^3)_{|T_xM_i}$ invariant. Since $G_i$ acts irreducibly on $T_xM_i$, the 2-form
$\Om_i$ has maximal rank.    We conclude that $M_i$ is a  K\"ahler manifold.  Denote by $\Lambda _i$ the bivector
in $\Lambda ^2 T_xM_i$ which is dual to $\Om_i$ with respect to the Riemannian metric. Since the K\"ahler form on $M_i$ is defined
uniquely up to a nonzero scalar multiple, the two-vector $\Lambda_i$ is defined uniquely up to a nonzero scalar multiple in the sense that $\Lambda_i$ does not depend on the  original vector $u$.  Now set $\hat u_i : = (\Lambda _i\rfloor \om ^3)_{|\R^p}$.  Let $\bar u_i \in T_x\R^p$ be the vector dual to $\hat u_i$ with respect to the given Riemannian metric. Then $\hat u_i (\bar u_i) = 1$. Clearly the  restriction of $\om^3$ to $M_i \times \la \bar u_i\ra _\R$ is a 3-form of the product type of maximal rank.  
This completes the proof of  Lemma \ref{product}.
\end{proof}

Let us complete the proof of Theorem \ref{3para}.   Set
$$\om^3_m: = \sum_i \pi_i^* (\bar \om^3_i) + \pi_0^* (\om^3_{|\R^p}),$$
where  $ i \in I\setminus J$ and $\pi_i: M\to M_i$, $\pi_0 : M \to \R^p$ are the  natural projections. 
Clearly  $\om ^3 -\om^3_m$ is also
invariant  under the action of $G$.
We note that the restriction of $\om^3-\om^3_m$ to each $M_i$ is zero, hence the rank of $\om^3 -\om^3_m$
to each $M_i$ has rank 2 by Lemma \ref{product}. Next we observe that the restriction of $\om^3 -\om^3_m$ to
$\R^p$ vanishes and it has rank 1, since in the opposite case, using the same argument as in the proof of Lemma \ref{product} we conclude that one of the space $M_i$, $i \in I\setminus J$, is flat, which contradicts our  assumption.

\begin{lemma}\label{final} We have $\om ^3-\om^3_m - \sum \Om^i \wedge \hat u_i = 0$.
\end{lemma}

\begin{proof} By our construction   the 3-form $\om ^3-\om^3_m - \sum_i \Om^i \wedge \hat u_i$ belongs  to the space 
$\pi_0 ^* ( \Om^3 (\R^p))$.  On the other hand, the restriction of $\om ^3-\om^3_m - \sum \Om^i \wedge \hat u_i$
to $\R^p$ is zero. This proves Lemma \ref{final}.
\end{proof}

 Clearly Lemma \ref{final}  completes the proof of Theorem \ref{3para}.
\end{proof}

\section* {Acknowledgement} 
This  note was partially supported by grant of ASCR Nr IAA100190701.   
The author is grateful to Sasha Elashvili, Willem Graaf, Todor Milev, Dmitri Panyushev,  Colleen   Robles, Andrew Swann, Ji\v ri
Van\v zura  for stimulating  and helpful discussions; especially  she is indebted to Ji\v ri
Van\v zura   for  turning  her interest  to this problem. She  thanks   anonymous referees for their   helpful remarks, which improved the  exposition  of this note and preventing the author to publish a wrong example.     She  thanks  Xiaobo Liu, Yoshihiro Ohnita, Hiroyuki Tasaki for  sending their reprints.  A part of this note has been written during the author stay at the ASSMS, GCU, Lahore-Pakistan.  She thanks ASSMS for their hospitality and financial support.

\medskip

Institute of Mathematics of ASCR, Zitna 25, 11567 Praha 1, Czech Republic, hvle@math.cas.cz\\
\end{document}